\documentclass[12pt,a4paper]{amsart}
\usepackage[margin=2.8cm]{geometry}
\usepackage{mathrsfs}%
\usepackage{amssymb}%
\usepackage{slashed}%

\newtheorem{theorem}{Theorem}[section]
\newtheorem{lemma}[theorem]{Lemma}
\newtheorem{corollary}[theorem]{Corollary}
\newtheorem{proposition}[theorem]{Proposition}

\theoremstyle{definition}
\newtheorem{definition}[theorem]{Definition}
\theoremstyle{remark}

\newtheorem{remark}[theorem]{Remark}

\begin{document}

\title[Characterizations of product Hardy spaces]{%
Characterizations of product Hardy spaces
on stratified groups
by singular integrals and maximal functions}

\author{Michael G. Cowling}
\address{School of Mathematics and Statistics, University of New South Wales, Sydney 2052, Australia}
\email{m.cowling@unsw.edu.au}
\author{Zhijie Fan}
\address{Department of Mathematics, Wuhan University, 430072, P. R. China}
\email{ZhijieFan@whu.edu.cn}
\author{Ji Li}
\address{Department of Mathematics, Macquarie University, NSW, 2109, Australia} \email{ji.li@mq.edu.au}
\author{Lixin Yan}
\address{Department of Mathematics, Sun Yat-sen (Zhongshan) University, Guang\-zhou, \phantom{M} 510275, P. R. China, and
Department of Mathematics, Macquarie University, NSW 2109, Australia}
\email{mcsylx@mail.sysu.edu.cn}

\date{}

%\maketitle

\begin{abstract}
A large part of the theory of Hardy spaces on products of Euclidean spaces has been extended to the setting of products of stratified Lie groups.
This includes characterisation of $\mathsf{H}^1$ by square functions and by atomic decompositions, proof of the duality of $\mathsf{H}^1$ with $\mathsf{BMO}$, and description of many interpolation spaces.
Until now, however, two aspects of the classical theory have been conspicuously absent: the characterisation of $\mathsf{H}^1$ by singular integrals (of Christ--Geller type) or by (vertical or nontangential) maximal functions.
In this paper we fill in these gaps by developing new techniques on products of stratified groups, using the ideas in \cite{CCLLO} on the Heisenberg group with flag structure.
\end{abstract}

\maketitle

\section{Introduction and statement of main results}

Hardy spaces first appeared in the study of boundary behaviour of holomorphic functions on the disc and upper half plane.
The modern theory of Hardy spaces began in 1960, when E.~M.~Stein and G.~Weiss \cite{SW60} considered functions defined on $\mathbb{R}^n \times \mathbb{R}^+$, and it took off in the early 1970s, with the remarkable work of C.~Fefferman and Stein \cite{FeS} and then R.~R.~Coifman and Weiss \cite{CW2}.
Much of this theory has been extended to more general spaces of homogeneous type, in the sense of Coifman and Weiss \cite{CW1, CW2}.
In the late 1970s, G.~B.~Folland and Stein \cite{FoSt} characterised the Hardy space $\mathsf{H}^1(G)$ on a stratified group $G$ in terms of atomic decompositions, square functions, area functions, and maximal functions.
The area integrals and maximal functions involve taking integrals or suprema over cones in $G \times \mathbb{R}_{+}$.
Soon after, M.~Christ and D.~Geller \cite{CG} showed that there are singular integral operators $\mathcal{R}_0, \dots, \mathcal{R}_n$ on a stratified Lie group such that $f \in \mathsf{H}^1(G)$ if and only if all $\mathcal{R}_jf \in \mathsf{L}^1(G)$.
Here $\mathcal{R}_0$ is the identity operator and the other $\mathcal{R}_j$ are Riesz transformations, that is, convolutions with derivatives of a potential.

Harmonic analysis on product spaces $\mathbb{R}^m \times \mathbb{R}^n$ was born in the late 1970s and studied extensively in the 1980s, in particular by S.-Y. A. Chang, R. Fefferman, R.~F.~Gundy, J.-L.~Journ\'e, J. Pipher, and Stein (see \cite{CF1, FSt, GS, Jo, KM, P}), motivated by problems on the boundary behavior of holomorphic functions in several complex variables, which require consideration of approach regions that behave differently in different variables.
Harmonic analysis on product spaces is influenced by classical harmonic analysis, but is different in that the different factors in the product may be dilated independently.
The terms one-parameter and multiparameter are often used to highlight the different structures of the dilations considered.
As in classical harmonic analysis, an important part of  the theory is the development of Hardy and BMO spaces, their duality and the connections to atomic decompositions.
A key ingredient is Journ\'e's covering lemma, which provides a tool to replace general open sets by rectangles with controlled geometry.

Since the 1980s, the development of multiparameter harmonic analysis proceeded apace;
recent contributions in the area include  \cite{fl, HyM, lppw, ops, Ou}.
Much of the product space theory on $\mathbb{R}^m \times \mathbb{R}^n$ has been extended to more general product spaces, including the duality of $\mathsf{H}^1$ with $\mathsf{BMO}$, characterisation of $\mathsf{H}^1$ by square functions and atomic decompositions, and description of various interpolation spaces.
In \cite{CDLWY, HLL, HLPW, HLW}, the theory of Hardy spaces $H^p$, for $p$ less than and close to $1$, has been developed on products $X_1\times X_2$ of spaces of homogeneous type.
Hence on products $G_1 \times G_2$ of stratified Lie groups, there is already a well-defined Hardy space $\mathsf{H}^1(G_1 \times G_2)$ that may be characterised by atomic decompositions and by square or area functions.

Two aspects of the classical theory that have been conspicuous by their absence until now are a singular integral characterisation of Christ--Geller type and a maximal function characterisation.
The main difficulty is that the geometrical structure is harder to handle than in the one-parameter case.
For example, one may obtain the atomic decomposition from the nontangential maximal function in the one-parameter case by using the classical Calder\'on--Zygmund and Whitney decompositions involving cubes, but
these decompositions are absent in the multiparameter case.

This paper fills these gaps for products of stratified Lie groups, with Theorems \ref{thm Riesz} and \ref{thm max} below.
For simplicity, and because new methods would otherwise be needed, we consider products of only two groups.
Our new techniques come from \cite{CCLLO}, where similar results are proved on the Heisenberg group with its flag structure.

Unexplained definitions may be found below.

\begin{theorem}\label{thm Riesz}
The double Riesz transformations $\mathcal{R}_{j_1}^{[1]} \otimes \mathcal{R}_{j_2}^{[2]}$ characterise the Hardy space $\mathsf{H}^1(G_1\times G_2)$.
That is, $f\in \mathsf{H}^1(G_1\times G_2)$ if and only if each $\mathcal{R}_{j_1}^{[1]} \otimes \mathcal{R}_{j_2}^{[2]} f$ is in $\mathsf{L}^{1}(G_1\times G_2)$, and moreover
\[
\Vert f \Vert_{\mathsf{H}^1(G_1\times G_2) }
\eqsim \sum_{j_1=0}^{d_1}\sum_{j_2=0}^{d_2}\left\Vert  \mathcal{R}_{j_1}^{[1]} \otimes \mathcal{R}_{j_2}^{[2]} f \right\Vert_{\mathsf{L}^{1}(G_1\times G_2)}.
\]
\end{theorem}

Using Theorem \ref{thm Riesz} and the $\mathsf{H}^1$-$\mathsf{BMO}$ duality (see for example \cite{HLL}), we obtain a decomposition of functions in the product space $\mathsf{BMO}( G_1\times G_2)$.

\begin{corollary}\label{cor BMO}
For a function $u$ on $G_1\times G_2$, the following are equivalent:
\begin{enumerate}
 \item[(a)] $u \in \mathsf{BMO}( G_1\times G_2) $;
 \item[(b)] there exist $g_{j_1, j_2}\in \mathsf{L}^\infty( G_1\times G_2 )$ such that
\[
u
=\sum_{j_1=0}^{d_1}\sum_{j_2=0}^{d_2} \mathcal{R}_{j_1}^{[1]} \otimes \mathcal{R}_{j_2}^{[2]}(g_{j_1, j_2}).
\]
\end{enumerate}
\end{corollary}

Write $\Gamma(g_1,g_2)$ for the product $\Gamma_1(g_1)\times \Gamma_2(g_2)$ of the cones treated by Folland and Stein \cite{FSt}, and for suitable functions $\psi^{[i]}$ on $G_i$, define the nontangential maximal function:
\[
\mathcal{N}_{\psi}(f)(g_1,g_2)
:=\sup \Bigl\{
\bigl| f \ast (\psi_{t_1}^{[1]}\otimes\psi_{t_2}^{[2]})(h_1,h_2) \bigr| :
(h_1,h_2) \in \Gamma(g_1,g_2), t_1, t_2 \in \mathbb{R}_{+} \Bigr\} ,
\]
where $\psi_{t_i}^{[i]}$ is a normalised dilate of $\psi^{[i]}$.

\begin{theorem}\label{thm max}
The nontangential maximal operator $\mathcal{N}_{\psi}$ characterises the Hardy space $\mathsf{H}^1(G_1\times G_2)$.
That is, $f\in \mathsf{H}^1(G_1\times G_2)$ if and only if $\mathcal{N}_{\psi}f$ is in $\mathsf{L}^{1}(G_1\times G_2)$; moreover
\[
\Vert f \Vert_{\mathsf{H}^1(G_1\times G_2) }
\eqsim \left\Vert  \mathcal{N}_{\psi}(f) \right\Vert_{\mathsf{L}^{1}(G_1\times G_2)}.
\]
\end{theorem}

This paper is organised as follows.
In Section 2, we remind the reader of some background on stratified Lie groups and analysis thereupon, and introduce some notation to simplify the formulae in the case of products of such groups.
In Section 3, we review some of the main results on Hardy spaces on products of stratified groups; many of these are valid in the more general context of products of spaces of homogeneous type.
Then we prove our main theorems on the characterisations of $\mathsf{H}^1(G_1 \times G_2)$, by Riesz transforms in Sections 4 and by maximal functions in Section 5.
The results proved are actually somewhat more general than stated in Theorems \ref{thm Riesz} and \ref{thm max}, but precise statements require more notation than we have established at this point.

``Constants'' are always positive real numbers;  we write $A \lesssim B$ when there is a constant $C$ such that $A \leq CB$, and $A \eqsim B$ when $A \lesssim B$ and $B \lesssim A$.
We denote the identity of a group by $o$, and the indicator function of a set $E$ by $\chi_{E}$.

\section{Preliminaries}

\subsection{Stratified nilpotent Lie groups}\label{ssec:strat-groups}
Let $G$ be a (real and finite dimensional) stratified nilpotent Lie group of step $k$ with Lie algebra $\mathfrak{g}$.
This means that we may write $\mathfrak{g}$ as a vector space direct sum $\mathfrak{v}_{1}  \oplus \dots \oplus \mathfrak{v}_{k}$, where $[\mathfrak{v}_1, \mathfrak{v}_{j}] = \mathfrak{v}_{j + 1}$ when $1\leq j \leq k$; here $\mathfrak{v}_{k+1} = \{0\}$.
Let $Q$ denote the homogeneous dimension $\sum_{j=1}^{k} j \dim \mathfrak{v}_{j}$ of $G$.

There is a one-parameter family of automorphic dilations $\delta_t$ on $\mathfrak{g}$, given by
\begin{align*}%\label{def-dilat}
\delta_t (X_1 + X_2+ \dots + X_{k}) = tX_1 + t^2X_2 + \dots + t^{k} X_{k};
\end{align*}
here each $X_j \in \mathfrak{v}_{j}$ and $t >0$.
The exponential mapping $\exp: \mathfrak{g} \to G$ is a diffeomorphism, and we identify $\mathfrak{g}$ and $G$.
The dilations extend to automorphic dilations of $G$, also denoted by $\delta_t$, by conjugation with $\exp$.
The natural bi-invariant Haar measure on $G$ is the Lebesgue measure on $\mathfrak{g}$, lifted to $G$ using $\exp$.

By \cite{HS}, the group $G$ may be equipped with a smooth subadditive homogeneous norm $\rho$, a continuous function from $G$ to $[0,\infty)$ that is smooth on $G\setminus \{o\}$
and satisfies
\begin{enumerate}
\item[(a)] $\rho(g^{-1}) =\rho(g)$;
\item[(b)] $\rho(xy) \leq \rho(x) + \rho(y)$
\item[(c)] $\rho({ \delta_t(g)}) =t\rho(g)$ for all $g\in G$ and $t>0$;
\item[(d)] $\rho(g) =0$ if and only if $g=o$,
\end{enumerate}
Abusing notation, we define $\rho(g, g') = \rho(g^{-1} g')$
for all $g, g' \in G$; this defines a metric on $G$.
We write $B(g, r)$ for the open ball with centre $g$ and radius $r$ with respect to $\rho$:
\[
B(g, r) = g B(o,r) = g \{ h \in G : \rho(h) < 1 \}.
\]
The metric space $(G,\rho)$ is \emph{geometrically doubling}; that is, there exists $N \in \mathbb{N}$ such that every metric ball $B(x,2r)$ may be covered by at most $N$ balls of radius $r$.

We remind the reader that a stratified Lie group is a space of homogenous type in the sense of Coifman and Weiss \cite{CW1, CW2}, and analysis on stratified Lie groups uses much from the theory of such spaces.
In particular, we frequently deal with \emph{molecules}, that is, functions $\psi$ that satisfy \emph{standard decay and smoothness conditions}, meaning that there is a parameter $\varepsilon \in (0,1]$, which we fix once and for all, and a constant $C$ such that
\begin{equation}\label{eq:molecule}
\begin{gathered}
\left| \psi(g) \right|  \leq C \frac{1}{(1+ \rho(g))^{Q+\varepsilon}} \\
\left| \psi(g)-\psi(g') \right|
\leq C \frac{\rho(g^{-1} g')^\varepsilon}{(1+ \rho(g)+ \rho(g'))^{Q+2\varepsilon} }
\end{gathered}
\end{equation}
for all $g, g' \in G$.
We often impose an additional \emph{cancellation} condition, namely
\begin{gather}
\int_G \psi(g) \,\mathrm{d}g =0 \label{cancel psi}.
\end{gather}
We write $\left\Vert \psi\right\Vert_{\mathsf{M}(G)}$ for the least constant $C$ such that the conditions \eqref{eq:molecule} hold, $\mathsf{M}(G)$ for the Banach space of all such functions $\psi$, and $\mathsf{M}_0(G)$ for the subspace of $\mathsf{M}(G)$ of all $\psi$ that also satisfy condition \eqref{cancel psi}.

The \emph{normalised dilate} $f_t$ of a function $f$ on $G$ by $t > 0$ is given by $f_{t} := t^{-Q}f\circ \delta_{1/t}$, and the \emph{convolution} $f\ast f'$ of measurable functions $f$ and $f'$ on $G$ is defined by
\begin{align*}
f \ast f'(g)
=\int_Gf(h)f'(h^{-1}g)\,\mathrm{d}h
=\int_Gf(gh^{-1})f'(h)\,\mathrm{d}h.
\end{align*}

Take left-invariant vector fields $\mathcal{X}_1$, \dots, $\mathcal{X}_{n}$ on $G$ that form a basis of $\mathfrak{v}_1$, and define the \emph{sub-Laplacian} $\mathcal{L} = -\sum_{j=1 }^{n} (\mathcal{X}_{j})^2 $.
Observe that each $\mathcal{X}_{j}$ is homogeneous of degree $1$ and $\mathcal{L}$ is homogeneous of degree $2$, in the sense that
\begin{align*}
&\mathcal{X}_{j} \left(  f \circ \delta_{t} \right)
= t \left(  \mathcal{X}_{j} f \right)  \circ \delta_{t} %\\
\qquad\text{and}\qquad
\mathcal{L} \left(  f \circ \delta_{t} \right)
= t^2 \, \left(  \mathcal{L} f \right)  \circ \delta_{t}
\end{align*}
for all $t > 0$ and all $f \in C^{2}(G)$.

Associated to the sub-Laplacian, there are various Riesz potential operators $\mathcal{L}^{-\alpha}$, where $\alpha > 0$; these are convolution operators with homogeneous kernels---see Folland \cite{F2}.
The \emph{Riesz transformation} $\mathcal{R}_{j} := \mathcal{X}_{j} \mathcal{L}^{-1/2}$ is a singular integral operator, and is bounded on $\mathsf{L}^{p}(G)$ when $1 < p < \infty$ as well as from the Folland--Stein Hardy space $\mathsf{H}^1(G)$ to $\mathsf{L}^1(G)$.
We define $\mathcal{R}_{0}$ to be the identity operator $\mathcal{I}$.

The Hardy--Littlewood maximal operator $\mathcal{M}$ on $G$ is defined using the metric balls:
\[
\mathcal{M}f(g)
:= \sup\left\{  \frac{1}{\left| B(g',r) \right| } \int_{B(g',r)} \left| f(g'') \right|  \,\mathrm{d}g'' : g \in B(g',r) \right\} .
\]
For future use, we note that the layer cake formula implies that, if $\mu$ is a radial decreasing function on $G$ (that is, $\mu(g)$ depends only on $\rho(g)$ and decreases as $\rho(g)$ increases), then
\begin{equation}\label{eq:decay-max}
\left| f \right|  \ast \mu_\varepsilon (g)
\leq  \left\Vert  \mu\right\Vert_{\mathsf{L}^1(G)} \mathcal{M}f(g)
\qquad\forall g \in G.
\end{equation}

\subsection{Functional calculus for the sub-Laplacian}\label{ssec:funct-calc}
The sub-Laplacian $\mathcal{L}$ has a spectral resolution:
\[
\mathcal{L} (f)=\int_0^{\infty}\lambda \,\mathrm{d}\mathcal{E}_{\mathcal{L} }(\lambda) f
\qquad\forall f\in \mathsf{L}^2(G),
\]
where $\mathcal{E}_{\mathcal{L}}(\lambda)$ is a projection-valued measure supported on $[0,\infty)$, the spectrum of $\mathcal{L}$.
For a bounded Borel function $\eta:[0,\infty)\to \mathbb{C}$, we define the operator $F(\mathcal{L})$ spectrally:
\begin{equation*}%\label{e111}
\eta(\mathcal{L})f = \int_0^{\infty} \eta(\lambda)\,\mathrm{d}\mathcal{E}_{\mathcal{L}}(\lambda) f
\qquad\forall f\in \mathsf{L}^2(G).
\end{equation*}
This operator is a convolution with a Schwartz distribution on $G$.

Take a smooth function $\eta:\mathbb{R}_{+} \to \mathbb{R}$, supported in $[1/2,2]$, such that $\sum_{n\in\mathbb{Z}}\eta(2^{-n}s)=1$ for all $s\in \mathbb{R}_{+}$.
The convolution kernels $k_{\eta(\mathcal{L}_i)}$ of the operators $\eta(\mathcal{L}_i)$ on $G$ are Schwartz functions, by \cite{Hul84}.
Moreover, we may write $\eta(t\mathcal{L}_i ) =t\mathcal{L}_i \psi(t\mathcal{L}_i )$, where $\psi(t) :=t^{-1}\eta(t)$ for all $t\in \mathbb{R}_{+}$ and $\operatorname{supp}\psi\subset[1/2,2]$, and deduce that
\begin{align*}
k_{\eta(t\mathcal{L}_i )} = t\mathcal{L}_i k_{\psi(t\mathcal{L}_i )}.
\end{align*}
Integration by parts now implies that
\begin{align*}%\label{eq:eta-Delta-cancel}
\int_{G }k_{\eta(t\mathcal{L}_i )}(g)\,\mathrm{d}g
=\int_{G }t\mathcal{L}_i k_{\psi(t\mathcal{L}_i )}(g)\,\mathrm{d}g
=0.
\end{align*}

\subsection{The heat and Poisson kernels}
Let $p_{t}$ and $P_{t}$, where $t > 0$, be the heat and Poisson kernels associated to the sub-Laplacian operator $\mathcal{L}$, that is, the convolution kernels of the operators $\mathrm{e}^{t\mathcal{L}}$ and $\mathrm{e}^{t \sqrt{\mathcal{L}}}$ on $G$.
We write $Q_{t}$ for $t\partial_t P_{t}$, and to simplify notation later, we often write $P$ instead of $P_{1}$.
We warn the reader that $P_{t}$ and $Q_{t}$ are the normalised dilates of $P_{1}$ and $Q_{1}$ by the factor $t$, but $p_{t}$ is the normalised dilate of $p_{1}$ by a factor of $t^{1/2}$.
Let $\nabla$ denote the subgradient on $G$ and $\slashed{\nabla}$ denote the gradient $(\nabla , \partial_{t})$ on $G\times\mathbb{R}_{+}$.

\begin{lemma}\label{lem:heat-pois-estim}
The kernels $p_{t}$ and $P_{t}$ are $\mathbb{R}_{+}$-valued.
Further, $p_{t}$ and $P_{t}$ have integral $1$, while $Q_{t}$ has integral $0$
for all $t \in \mathbb{R}_{+}$.
Finally, there exists a constant $c$ such that
\begin{align*}%\label{GE}
p_{t}(g)
&\lesssim t^{-Q/2}\exp\left( -{\rho^{2}(g)}/{ct}\right)
\\
\left| \slashed{\nabla} p_{t}(g) \right|
&\lesssim t^{-(Q+1)/2}\exp\left( -{\rho^{2}(g)}/{ct}\right)
\\
P_{t}(g)
&\eqsim \frac{ t }{ (t^2 + \rho(g)^2)^{(Q + 1)/2}}
\\
\left| \slashed{\nabla} P_{t}(g) \right|
&\lesssim \frac{ t }{ (t^2 + \rho(g)^2)^{(Q + 2)/2}}
\end{align*}
for all $g\in G$ and $t \in \mathbb{R}_{+}$.
\end{lemma}
\begin{proof}
For the heat kernel estimates, see \cite[Theorem IV.4.2]{VSC92}.
Note that there is a version of the first estimate with the opposite inequality and a different constant $c$.

The estimates for $P_{t}$ and $Q_{t}$ follow from the subordination formula
\begin{equation*}%\label{eq:subord}
\mathrm{e}^{-t\sqrt{\mathcal{L}}}
=\frac{1}{2\sqrt{\pi}}\int_0^\infty
\frac{t\mathrm{e}^{-{t^2}/{4v}}}{\sqrt{v}}\mathrm{e}^{-v \mathcal{L}}\frac{dv}{v}.
\end{equation*}
For the case of the Heisenberg group, much of this is worked out in detail in \cite{CCLLO}.
\end{proof}

This lemma implies that the heat kernel $p_{1}$ and the Poisson kernel $P_{1}$ (and their derivatives) both satisfy the standard decay and smoothness conditions \eqref{eq:molecule}; the derivatives also satisfy the cancellation condition \eqref{cancel psi}.

Lemma \ref{lem:heat-pois-estim} also implies the following standard corollary, whose proof we omit.

\begin{corollary}\label{cor:Lebesgue}
Suppose that $f \in \mathsf{L}^{p}(G)$, where $1 \leq p \leq \infty$.
Then $\left\Vert  f \ast P_{t} \right\Vert_{\mathsf{L}^{p}(G)}$ and $\left\Vert   f \ast t\slashed{\nabla} P_{t} \right\Vert_{\mathsf{L}^{p}(G)}$ are uniformly bounded as $t$ runs over $\mathbb{R}_{+}$.
Further,
\[
\lim_{t \to 0} f \ast P_{t} = f ;
\]
the convergence is both pointwise almost everywhere, and in the $\mathsf{L}^{p}(G)$ norm if $1 \leq p < \infty$ and in the weak-star topology if $p = \infty$.
Finally,
\begin{align*}
\lim_{t \to 0} f \ast t\slashed{\nabla} P_{t}=0;
\end{align*}
the convergence is both pointwise almost everywhere, and in the strong operator topology if $f \in \mathsf{L}^1(G)$, in the $\mathsf{L}^{p}(G)$ norm if $1 < p < \infty$ and in the weak-star topology if $p = \infty$.
\end{corollary}

\subsection{Systems of pseudodyadic cubes}
\label{sec:pseudodyadiccubes}

We use the Hyt\"onen--Kairema \cite{HK} families of ``dyadic cubes" in geometrically doubling metric spaces.
We state a version of \cite[Theorem 2.2]{HK} that is simpler, in that we work on well-behaved metric spaces rather than general pseudometric spaces.
The Hyt\"onen--Kairema construction builds on seminal work of Christ \cite{Chr} and of Sawyer and Wheeden \cite{SW}.

\begin{theorem}[\cite{HK}]
\label{thm:pseudodyadiccubes}
Let $(G,\rho)$ be a metric stratified group  and $c_0$, $C_0$ and $\kappa$ constants such that $0 < c_0 \leq C_0 < \infty$ and $12 C_0\kappa \leq c_0$. Then for all $k \in \mathbb{Z}$, there exist families $\mathscr{Q}^k(G)$ of \emph{pseudodyadic cubes} $Q$ with \emph{centres} $z(Q)$, such that:
\begin{enumerate}
\item[(a)] $G$ is the disjoint union of all $Q \in \mathscr{Q}^k(G)$, for each $k\in\mathbb{Z}$;
\item[(b)] $B(z(Q),c_0\kappa^k/3)\subseteq Q \subseteq B(z(Q),2C_0\kappa^k)$ for all $Q \in \mathscr{Q}^k(G)$;
\item[(c)] if $Q \in \mathscr{Q}^k(G)$ and $Q' \in \mathscr{Q}^{k'}(G)$ where $k\leq k'$, then either $Q \cap Q'=\emptyset$ or $Q \subseteq Q'$; in the second case, $B(z(Q), 2C_0\kappa^k) \subseteq B(z(Q'),2C_0\kappa^{k'})$;
\end{enumerate}
\end{theorem}

The family of pseudodyadic cubes $Q$ in $\mathscr{Q}^k(G)$, where $k \in \mathbb{Z}$, of  Theorem~\ref{thm:pseudodyadiccubes} will be called a \emph{Hyt\"onen--Kairema set of cubes} on $G$.
We write $\mathscr{Q}(G)$ for the union of all $\mathscr{Q}^k(G)$.
Given a cube $Q \in \mathscr{Q}^k(G)$, we denote the quantity $\kappa^k$ by $\ell(Q)$, by analogy with the side-length of a Euclidean cube.

\subsection{Products of stratified groups}

We equip products of stratified groups $G_1$ and $G_2$ with a product structure: the basic geometric objects are rectangles, which are products of balls, and pseudodyadic rectangles, which are products of pseudodyadic cubes.
We write $\mathscr{P}^{\boldsymbol{j}}(\boldsymbol{G})$ for the collection of all pseudodyadic rectangles that are products of cubes in $\mathscr{Q}^{j_1}(G_1)$ and in $\mathscr{Q}^{j_2}(G_2)$; $\mathscr{P}(\boldsymbol{G})$ for the collection of all pseudodyadic rectangles, and $\mathscr{R}(\boldsymbol{G})$ for the collection of all rectangles.
We let $\boldsymbol{\ell}: \mathscr{P}(\boldsymbol{G}) \to \boldsymbol{T}$ be the function such that $\ell_i(Q_1 \times Q_2) = \ell(Q_i)$, the ``side-length'' of $Q_i$.

We carry forward the notation from Section \ref{ssec:strat-groups}, modified by adding a subscript $i$ or superscript $[i]$ to clarify that we are dealing with $G_i$.
To shorten the formulae, we often use bold face type to indicate a product object: thus we write $\boldsymbol{G}$, $\boldsymbol{g}$, $\boldsymbol{r}$ and $\boldsymbol{t}$ in place of $G_1 \times G_2$, $(g_1,g_2)$, $(r_1,r_2)$ and $(t_1,t_2)$.
For example, $B_i (g_i, r_i)$ denotes the open ball on $G_i$ with centre $g_i$ and radius $r_i$, with respect to the homogeneous norm $\rho_i $, and a typical rectangle $R(\boldsymbol{g},\boldsymbol{r})$ is then a product $B_1(g_1, r_1) \times B_2(g_2, r_2)$.
We also write $\boldsymbol{t} \,\mathrm{d}\boldsymbol{t}$ in place of $t_1 t_2 \,\mathrm{d}t_1 \,\mathrm{d}t_2$, and $\boldsymbol{T}$ for the product parameter space $\mathbb{R}_+ \times \mathbb{R}_+$.

The element of Haar measure on $\boldsymbol{G}$ is denoted $\mathrm{d}\boldsymbol{g}$, but may be written as $\mathrm{d}g_1\,\mathrm{d}g_2$ for calculations.
The convolution $f\ast f'$ of functions $f$ and $f'$ on $\boldsymbol{G}$ is defined by
\begin{align*}
(f\ast f')(\boldsymbol{g})
:=\int_{\boldsymbol{G}}f(\boldsymbol{h})f'(\boldsymbol{h}^{-1}\boldsymbol{g}) \,\mathrm{d}\boldsymbol{h}.
\end{align*}

We define the strong maximal operator $\mathcal{M}_{S}$ by
\begin{align*}
\mathcal{M}_{S}(f)(\boldsymbol{g})
:=\sup\left\{  \frac{1}{\left| R\right| }\int_{R}\left| f(\boldsymbol{h}) \right| \,\mathrm{d}\boldsymbol{h} :
R \ni \boldsymbol{g}, R \in \mathscr{R}(\boldsymbol{G}) \right\} .
\end{align*}
It is a straightforward exercise to show that $\mathcal{M}_{S}$ is dominated by the iterated Hardy--Littlewood maximal operators in the factors:
\[
\mathcal{M}_{S}{f}
\leq \mathcal{M}_1 \mathcal{M}_2 (f)
\qquad\text{and}\qquad
\mathcal{M}_{S}{f}
\leq \mathcal{M}_2 \mathcal{M}_1(f)
\qquad\forall f \in \mathsf{L}^1_{\mathrm{loc}}(\boldsymbol{G}).
\]
When $1 < p \leq \infty$, the operators $\mathcal{M}_1$ and $ \mathcal{M}_2$ in the factors are $\mathsf{L}^{p}$-bounded, so the iterated maximal operators and the strong maximal operator are also $\mathsf{L}^{p}$-bounded.

Given functions $\psi^{[1]}$ on $G_1$ and $\psi^{[2]}$ on $G_2$, we often deal with the product of their normalised dilates on $G_1 \times G_2$, and we abbreviate this to $\psi_{\boldsymbol{t}}$:
\[
\psi_{\boldsymbol{t}} := \psi^{[1]}_{t_1} \otimes \psi^{[2]}_{t_2} \,.
\]
If $\psi^{[1]} \in \mathsf{M}(G_1)$ and $\psi^{[2]} \in \mathsf{M}(G_2)$, then
\begin{equation*}%\label{eq:Poisson-leq-strong}
\left|  f \ast \psi_{\boldsymbol{t}}(\boldsymbol{g})  \right|  \lesssim \mathcal{M}_{S}(f)(\boldsymbol{g})
\qquad\forall \boldsymbol{g} \in \boldsymbol{G} \quad\forall f \in \mathsf{L}^1(\boldsymbol{G}),
\end{equation*}
much as argued to prove \eqref{eq:decay-max}, but with ``biradial'' in place of ``radial''.

Given an open subset $U$ of $\boldsymbol{G}$ with finite measure $\left| U \right| $, we define the enlargement $\widetilde{U}$ of $U$ using the strong maximal operator $\mathcal{M}_{S}$:
\begin{align*}
\widetilde{U}
:= \Bigl\{ \boldsymbol{g} \in \boldsymbol{G} : \mathcal{M}_{S} \chi_{U}(\boldsymbol{g}) > \frac{1}{4} \Bigr\} .
\end{align*}
We write $\mathscr{M}(U)$ for the family of maximal pseudodyadic rectangles contained in $U$.

We let $P_{\boldsymbol{t}}:=P^{[1]}_{t_1} \otimes P^{[2]}_{t_2}$; when $t_1 = 0$ or $t_2=0$, we interpret this as a distribution supported in $G_2$ or in $G_1$ in the obvious way.
We write $Q^{[i]}_{t_i}$ for the convolution kernel of the operator $t_i \partial_{t_i} \mathrm{e}^{-t_i \sqrt{\mathcal{L}_i}}$; then $Q^{[i]}_{t_i} = t_i \partial_{t_i} P^{[i]}_{t_i}$.
By arguing as in Corollary \ref{cor:Lebesgue}, it is easy to see that for any measurable subset $V$ of $\boldsymbol{G}$,
\begin{equation}\label{eq:chi-star-Q}
\lim_{t_1\to 0} \chi_V \ast (Q^{[1]}_{t_1} \otimes P^{[2]}_{t_2})(\boldsymbol{g}) = 0
\end{equation}
for almost all $\boldsymbol{g}$ in $\boldsymbol{G}$ and in the weak-star topology of $\mathsf{L}^\infty(\boldsymbol{G})$.

The double Riesz transforms $\mathcal{R}^{[1]}_{j_1} \otimes \mathcal{R}^{[2]}_{j_2}f$, where $0 \leq j_i \leq d_i$, of a suitable function $f$ on $\boldsymbol{G}$ are defined in the obvious way:  when $j_1$ and $j_2$ are nonzero,
\begin{equation}\label{eq:def-double-Riesz}
\mathcal{R}^{[1]}_{j_1} \otimes \mathcal{R}^{[2]}_{j_2}f
:= X^{[1]}_{j_1} \mathcal{L}_1^{-1/2} X^{[2]}_{j_2} \mathcal{L}_2^{-1/2}f,
\end{equation}
and if $j_i = 0$ we replace $X^{[i]}_{j_i} \mathcal{L}_i^{-1/2}$ by the identity operator $\mathcal{I}_i$.

\section{The known product Hardy spaces}

\subsection{The atomic Hardy space}

Fix a constant $C$ and Hyt\"onen--Kairema sets of pseudo\-dyadic cubes in $G_1$ and $G_2$.
A pseudodyadic rectangle $R$ is a product $Q_1\times Q_2$ of pseudo\-dyadic cubes in the factors $G_1$ and $G_2$.

An integrable function $a_R$ is said to be a \emph{particle} associated to the pseudodyadic rectangle $R$ if the following support and product cancellation conditions hold:
\begin{equation}\label{eq:product-atom-support}
\operatorname{supp} a_R\subseteq CR
\end{equation}
and
\begin{equation}\label{eq:product-atom-cancel}
\int_{G_1}a_R(g_1,\cdot) \,\mathrm{d}g_1=0
\qquad\text{and}\qquad
\int_{G_2}a_R(\cdot,g_2) \,\mathrm{d}g_2=0
\end{equation}
(almost everywhere).

A function $a$ on $\boldsymbol{G} $ is said to be a \emph{product atom} associated to an open subset $U$ of $\boldsymbol{G}$ of finite measure if $a$ satisfies the following support and size conditions:
\[
\operatorname{supp} a\subset\widetilde{U}
\]
\[
\left\Vert  a \right\Vert_{\mathsf{L}^2( \boldsymbol{G} )}\leq \left| \widetilde{U} \right| ^{-1/2} ,
\]
and we may decompose $a$ as a sum $\sum_{R\in \mathscr{M}(U)}a_R $ of particles $a_R$ associated to the pseudodyadic rectangles $R \in \mathscr{M}(U)$ in such a way that
\[
\biggl(\sum_{R\in \mathscr{M}(U)}\left\Vert  a_R \right\Vert_{\mathsf{L}^2(\boldsymbol{G} )}^2\biggr)^{1/2}
\leq \left| {U} \right|  ^{-1/2}.
\]

\begin{definition}\label{def:atom-Hardy}
We say that $f \in \mathsf{L}^1(\boldsymbol{G})$ belongs to the atomic Hardy space $\mathsf{H}^1_{\mathrm{atom}} (\boldsymbol{G})$ if and only if it is possible to represent $f$ as a sum
\[
f = \sum_{n\in\mathbb{N}} \lambda_n a_n,
\]
where $a_n$ is an atom and $\lambda_n \in \mathbb{R}_{+}$ for all $n$, and $\sum_{n\in\mathbb{N}} \lambda_n < \infty$.
We define the norm $\Vert f \Vert_{\mathsf{H}^1_{\mathrm{atom}}(\boldsymbol{G})}$ to be the infimum of the sums $\sum_{n\in\mathbb{N}} \lambda_n$ over all such representations of $f$.
\end{definition}

It is often more convenient to impose a stronger requirement on particles, namely, that $a_R = \mathcal{L}_1^{N_1} \mathcal{L}_1^{N_2}b_R$ for some $\mathsf{L}^2(\boldsymbol{G})$ function $b_R$ in the domain of $\mathcal{L}_1^{N_1} \mathcal{L}_1^{N_2}$ and for large integers $N_1$ and $N_2$; this means that $a_R$ has many vanishing moments, which may make calculations easier.
We may show that this stronger requirement on particles gives the same atomic Hardy space, using telescopic series arguments to make moments vanish.

\subsection{Square function and area function Hardy spaces}\label{ssec:sq-fn-area-fn-Hardy}

For $\boldsymbol{g}\in \boldsymbol{G}$ and $\beta\in [0,\infty)$, we write $\Gamma^{\beta}(\boldsymbol{g})$ for the product cone $\Gamma_1^{\beta}(g_1)\times \Gamma_2^{\beta}(g_2)$, where
\[
\Gamma_i^{\beta}(g_i)
:=\{(h_i ,t_i )\in G_i \times \mathbb{R}_{+}:\rho_i (g_i ,h_i ) \leq \beta t_i \}.
\]
We work on the domain $G_1\times G_2 \times \mathbb{R}_{+} \times \mathbb{R}_{+}$.

Take functions $\psi^{[i]}$ on $G_i$ that satisfy the standard decay, smoothness and cancellation conditions \eqref{eq:molecule} and \eqref{cancel psi}.
Recall that $\psi_{\boldsymbol{t}}$ denotes the product function $\psi^{[1]}_{t_1} \otimes \psi^{[2]}_{t_2}$.

\begin{definition}\label{def:sq-fn}
For $\psi^{[i]}$ as above and $\beta > 0$, we define $\mathcal{S}_{\psi,\beta}(f)(\boldsymbol{g})$ to be
\[
\left( \iint_{\Gamma^{\beta}(\boldsymbol{g})}
\frac{\left| (f \ast \psi_{\boldsymbol{t}}(\boldsymbol{h}) \right| ^{2}}{\left| R(\boldsymbol{o},\beta \boldsymbol{t}) \right|  }
\,\mathrm{d}\boldsymbol{h} \,\frac{\mathrm{d}\boldsymbol{g}}{\boldsymbol{t}} \right) ^{1/2}
\]
for all $\boldsymbol{g} \in \boldsymbol{G}$ and $f \in \mathsf{L}^1(\boldsymbol{G})$.
We also define
\[
\mathcal{S}_{\psi,0}(f)(\boldsymbol{g})
:=\left( \int_{\boldsymbol{T}} \left| f \ast \psi_{\boldsymbol{t}}(\boldsymbol{g}) \right| ^{2} \,\frac{\mathrm{d}\boldsymbol{g}}{\boldsymbol{t}}\right) ^{1/2}
\]
for all $\boldsymbol{g} \in \boldsymbol{G}$ and $f \in \mathsf{L}^1(\boldsymbol{G})$.
The Hardy space $\mathsf{H}^1_{\mathrm{sq},\psi,\beta}(\boldsymbol{G})$ is defined to be the space
\[
\{f\in \mathsf{L}^{1}(\boldsymbol{G}): \left\Vert  \mathcal{S}_{\psi,\beta}(f) \right\Vert_{\mathsf{L}^{1}(\boldsymbol{G})} <\infty\},
\]
equipped with the norm
\[
\Vert f \Vert_{\mathsf{H}^1_{\mathrm{sq},\psi,\beta}(\boldsymbol{G})} :=\left\Vert  \mathcal{S}_{\psi,\beta}(f) \right\Vert_{\mathsf{L}^{1}(\boldsymbol{G})}.
\]
\end{definition}

Note that $\mathcal{S}_{\psi,\beta}(f)$ tends to $\mathcal{S}_{\psi,0}(f)$ as $\beta \to 0$, at least pointwise.
There are also discrete versions of this definition, where the integrals over $\mathbb{R}_{+}$ are replaced by sums over powers of $2$ (or some other base).
We usually call $\mathcal{S}_{\psi,\beta}(f)$ an area function when $\beta > 0$ and a square function when $\beta = 0$, but it is more efficient to treat these together.

As mentioned earlier, much is known about Hardy spaces defined as above, and we summarise some of the main results.
From \cite{HLL}, the space $\mathsf{H}^1_{\mathrm{sq},\psi,0}(\boldsymbol{G})$ is independent of the choice of the functions $\psi^{[i]}$, provided that they satisfy the decay, smoothness and cancellation conditions \eqref{eq:molecule} and \eqref{cancel psi}; discrete square functions and area operators $\mathcal{S}_{\psi,1}$ also characterise the same space, which we write simply as $\mathsf{H}^1(\boldsymbol{G})$.
The key technique to prove these equivalences is a Plancherel--P\'olya inequality.
From \cite{HLPW} and \cite{CDLWY}, we see also that $\mathsf{H}^1(\boldsymbol{G})$ may be characterised using wavelet and atomic decompositions; more precisely, $\mathsf{H}^1_{\mathrm{atom}}(\boldsymbol{G}) =\mathsf{H}^1(\boldsymbol{G})$.
Further, the double Riesz transformations $\mathcal{R}^{[1]}_{j_1}\otimes \mathcal{R}^{[1]}_{j_1}$ (see Definition \ref{eq:def-double-Riesz}) and similar singular integral operators are all bounded from $\mathsf{H}^1(\boldsymbol{G})$ to $\mathsf{L}^{1}(\boldsymbol{G})$.
Finally, from \cite{HLL}, the dual of $\mathsf{H}^1(\boldsymbol{G})$ is the space $\mathsf{BMO}$ defined in terms of (suitable product) Carleson measures on $\boldsymbol{G}$.

In Section \ref{ssec:cone-angle} below, we show that the space $\mathsf{H}^1_{\mathrm{sq},\psi,\beta}(\boldsymbol{G})$ is also independent of $\beta$.

Let $\nabla_i$ and $\mathcal{L}_i $ denote the subgradient and the sub-Laplacian on $ G_i$; recall that $\slashed{\nabla}_i$ denotes the gradient $(\nabla_i , \partial_{t})$ on $G_i\times\mathbb{R}_{+}$.
The (vector-valued) convolution kernels of the operators $t_i\mathcal{L}_i \mathrm{e}^{-t_i\mathcal{L}_i }$ and $t_i \slashed{\nabla}_i \mathrm{e}^{-t_i \sqrt{\mathcal{L}_i }}$ satisfy the decay, smoothness and cancellation conditions \eqref{eq:molecule} and \eqref{cancel psi}.
Hence $\mathsf{H}^1(\boldsymbol{G})$ may also be characterised via the Littlewood--Paley area functions and square functions defined using the heat and Poisson kernels.

\subsection{Independence of cone angle}\label{ssec:cone-angle}

Recall that $R(\boldsymbol{g},\boldsymbol{t}):=B_1(g_1,t_1)\times B_2(g_2,t_2)$.
Fix a parameter $\theta$ in $(0,1)$.

If $V$ is a closed subset of $\boldsymbol{G}$, then we say that $\boldsymbol{g}\in \boldsymbol{G}$ has global $\theta$-density with respect to $V$ if
\[
\frac{\left| V\cap R(\boldsymbol{g},\boldsymbol{t}) \right| }{\left| R(\boldsymbol{g},\boldsymbol{t}) \right| } \geq \theta
\]
for all $\boldsymbol{t} \in \boldsymbol{T}$.
Let $V^{*}$ be the set containing all points of global $\theta$-density of $V$, then $V^{*}$ is closed and $V^{*}\subseteq V$.
Equivalently,
\[
(V^*)^c =\{\boldsymbol{g} \in \boldsymbol{G}:\mathcal{M}_{S}(\chi_{V^c})(\boldsymbol{g}) > 1-\theta\}.
\]
It follows from the $\mathsf{L}^{}\log \mathsf{L}^{} \to \mathsf{L}^{1,\infty}$ estimate for the strong maximal function (see, for example, \cite{CF}) that $\left| (V^{*})^{c} \right| \leq c_{\theta}\left| V^{c} \right| $, where
\[
c_{\theta}
=\frac{C}{1-\theta}\left( 1+\log_2^{+}\left( \frac{1}{1-\theta}\right) \right) .
\]

For a closed subset $V$ of $\boldsymbol{G}$, write
\[
W^{\beta}(V):=\bigcup_{\boldsymbol{g}\in V}\Gamma^{\beta}(\boldsymbol{g}).
\]
\begin{lemma}\label{sfsfsf}
Suppose that $V$ is a closed set in $\boldsymbol{G}$ such that $\left| V^{c} \right| <\infty$.
Then there exist constants $c_{0}\leq 1/4$ and $C$ such that if $\beta>1$ and $\theta=1- c_{0}\beta^{-Q_1-Q_2}$, then
\begin{align*}
\iint_{W^{\beta}(V^{*})} F(\boldsymbol{g},\boldsymbol{t}) \left| R(\boldsymbol{o},\boldsymbol{t}) \right|  \,\mathrm{d}\boldsymbol{g}\,\mathrm{d}\boldsymbol{t}
&\lesssim \int_{V}\iint_{\Gamma(\boldsymbol{g})} F(\boldsymbol{h} ,\boldsymbol{t})
\,\mathrm{d}\boldsymbol{h} \,\mathrm{d}\boldsymbol{t}\,\mathrm{d}\boldsymbol{g}
\end{align*}
for all measurable nonnegative-real-valued functions $F$ on $\boldsymbol{G}\times \boldsymbol{T}$.
\end{lemma}
\begin{proof}
First, if $(\boldsymbol{h}, \boldsymbol{t})\in W^{\beta}(V^{*})$, then there exists $\tilde{\boldsymbol{g}} \in V^{*} \cap R(\boldsymbol{h},\beta\boldsymbol{t})$.
We see easily that
\begin{align*}
\left| R(\tilde{\boldsymbol{g}},\beta \boldsymbol{t})\cap R(\boldsymbol{h},\boldsymbol{t})^{c} \right|
\leq \left( 1- 2c_{0}\beta^{-Q_1-Q_2}\right) \left| R(\tilde{\boldsymbol{g}},\beta \boldsymbol{t}) \right| ,
\end{align*}
for some constant $c_{0}\leq 1/4$.
Hence
\begin{align*}
\left| V\cap R(\boldsymbol{h},\boldsymbol{t}) \right|
&\geq \left| V\cap R(\tilde{\boldsymbol{g}},\beta \boldsymbol{t}) \right|
-\left| R(\tilde{\boldsymbol{g}},\beta \boldsymbol{t})\cap R(\boldsymbol{h},\boldsymbol{t})^{c} \right| \\
&\geq \left( \theta-1+2c_{0}\beta^{-Q_1-Q_2}\right) \left| R(\tilde{\boldsymbol{g}},\beta \boldsymbol{t}) \right| \\
&= c_{0}\beta^{-Q_1-Q_2} \left| R(\tilde{\boldsymbol{g}},\beta \boldsymbol{t}) \right|
\geq C \left| R(\boldsymbol{g},\boldsymbol{t}) \right|  .
\end{align*}

Now, by Fubini's Theorem,
\begin{align*}
&\int_{V}\iint_{\Gamma(\boldsymbol{g})} F(\boldsymbol{h},\boldsymbol{t}) \,\mathrm{d}\boldsymbol{h}\,\mathrm{d}\boldsymbol{t}\,\mathrm{d}\boldsymbol{g}\\
&=\iint_{\boldsymbol{T} \times \boldsymbol{G}}
\int_{V} \chi_{R(\boldsymbol{o},\boldsymbol{t})}
\left( \boldsymbol{h}^{-1}\boldsymbol{g}\right)
F(\boldsymbol{h},\boldsymbol{t})
\,\mathrm{d}\boldsymbol{h}  \,\mathrm{d}\boldsymbol{g}\,\mathrm{d}\boldsymbol{t}\\
&\geq\iint_{W^{\beta}(V^{*})}\int_{\boldsymbol{G}}
\chi_{R(\boldsymbol{o},\boldsymbol{t})}
\left( \boldsymbol{h}^{-1}\boldsymbol{g}\right)  F(\boldsymbol{h},\boldsymbol{t})
\,\mathrm{d}\boldsymbol{g}\,\mathrm{d}\boldsymbol{h}\,\mathrm{d}\boldsymbol{t}\\
&=\iint_{W^{\beta}(V^{*})} F(\boldsymbol{h},\boldsymbol{t})
\left| R(\boldsymbol{o},\boldsymbol{t}) \right|  \,\mathrm{d}\boldsymbol{h}\,\mathrm{d}\boldsymbol{t},
\end{align*}
as required.
\end{proof}

\begin{proposition}\label{prop equiv Hardy different beta}
With the notation of Definition \ref{def:sq-fn},
\[
\mathsf{H}^1_{\mathrm{sq},\psi,\beta}(\boldsymbol{G}) = \mathsf{H}^1_{\mathrm{sq},\psi,1}(\boldsymbol{G}) ,
\]
and these spaces have equivalent norms for all $\beta > 0$.
\end{proposition}

\begin{proof}
It suffices to suppose that $\beta>1$ and show that
\[
\left\Vert  \mathcal{S}_{\psi,\beta}(f) \right\Vert_{\mathsf{L}^1(\boldsymbol{G})}
\lesssim \beta^{Q_1+Q_2}(1+\log_2^+\beta)\left\Vert  \mathcal{S}_{\psi,1}(f) \right\Vert_{\mathsf{L}^1(\boldsymbol{G})}.
\]
For all $\lambda>0$, set
\begin{align*}
V=\{\boldsymbol{g} \in \boldsymbol{G} :\mathcal{S}_{\psi,1}(f)(\boldsymbol{g})\leq \lambda\},
\end{align*}
and $\theta=1- \beta^{-Q_1-Q_2}/4$.
Then, from Lemma \ref{sfsfsf} and Fubini's theorem,
\begin{align*}
\int_{V^{*}}\mathcal{S}_{\psi,\beta}(f)(\boldsymbol{g})^{2}\,\mathrm{d}\boldsymbol{g}
&=\int_{V^{*}}\iint_{\Gamma^{\beta}(\boldsymbol{g})}
\frac{\left| f\ast \psi_{\boldsymbol{t}} (\boldsymbol{h}) \right| ^{2} }{\left| R(\boldsymbol{o},\beta \boldsymbol{t}) \right| }
\,\mathrm{d}\boldsymbol{h} \,\frac{\mathrm{d}\boldsymbol{g}}{\boldsymbol{t}}\,\mathrm{d}\boldsymbol{g}\\
&\lesssim \beta^{Q_1+Q_2}
\iint_{W^{\beta}(V^{*})}\left| f\ast \psi_{\boldsymbol{t}} (\boldsymbol{h}) \right| ^{2} \,\mathrm{d}\boldsymbol{h}\,\frac{\mathrm{d}\boldsymbol{g}}{\boldsymbol{t}}\\
&\lesssim\beta^{Q_1+Q_2}
\int_{V}\iint_{\Gamma(\boldsymbol{g})}
\frac{\left| f\ast \psi_{\boldsymbol{t}} (\boldsymbol{h}) \right| ^{2}}
{\left| R(\boldsymbol{o},\beta \boldsymbol{t}) \right| }
\,\mathrm{d}\boldsymbol{h}\,\frac{\mathrm{d}\boldsymbol{g}}{\boldsymbol{t}}
\,\mathrm{d}\boldsymbol{g}\\
&\eqsim\beta^{Q_1+Q_2}\int_{V}\mathcal{S}_{\psi,1}(f)^{2}\,\mathrm{d}\boldsymbol{g}.
\end{align*}
Therefore
\begin{align*}
&\left| \{\boldsymbol{g}\in \boldsymbol{G} : \mathcal{S}_{\psi,\beta}(f)(\boldsymbol{g}) > \lambda\} \right|  \\
&\qquad\leq \left| (V^{*})^c \right|  + \frac{C}{\lambda^{2}} \int_{V^{*}}\mathcal{S}_{\psi,\beta}(f)(\boldsymbol{g})^{2}\,\mathrm{d}\boldsymbol{g}\\
&\qquad\leq C\beta^{Q_1+Q_2} (1 + \log_2^{+}\beta)
\left( \left| (V^{*})^c \right|  +\frac{1}{\lambda^{2}}
\int_{V}\mathcal{S}_{\psi,1}(f)(\boldsymbol{g})^{2} \,\mathrm{d}\boldsymbol{g} \right) .
\end{align*}
Integrating with respect to $\lambda$ yields
\[
\left\Vert  \mathcal{S}_{\psi,\beta}(f) \right\Vert_{\mathsf{L}^1(\boldsymbol{G})}
\lesssim \beta^{Q_1+Q_2}(1+\log_2^{+}\beta)\left\Vert  \mathcal{S}_{\psi,1}(f) \right\Vert_{\mathsf{L}^1(\boldsymbol{G})} ,
\]
which completes the proof of Proposition \ref{prop equiv Hardy different beta}.
\end{proof}

\subsection{Summary}
The known results cited in Section \ref{ssec:sq-fn-area-fn-Hardy} and our additional material here may be summarised in the following proposition.

\begin{proposition}\label{prop equiv Hardy via L-P}
The atomic Hardy space $\mathsf{H}^1_{\mathrm{atom}}(\boldsymbol{G})$ and the square function and area function Hardy spaces $\mathsf{H}^1_{\mathrm{sq},\psi,\beta}$ for different $\psi$ and $\beta$ coincide and have equivalent norms.
\end{proposition}

\section{The singular integral characterisation}

We consider a stratified Lie group $G$.
Recall that $\mathcal{R}_{0}$ is the identity operator $\mathcal{I}$, and when $1 \leq j \leq d_i$, the $j$th Riesz operator $\mathcal{R}_{j}$ on $G$ is defined by
\[
\mathcal{R}_{j}:=\mathcal{X}_{j} (\mathcal{L})^{-{1}/{2}};
\]
its convolution kernel, $k_{j}$ say, is smooth away from the identity of $G$, and homogeneous of degree $-Q$.
According to Christ and Geller \cite{CG}, $f \in \mathsf{H}^1(G)$ if and only if all $\mathcal{R}_{j} f \in \mathsf{L}^1(G)$, and there is a corresponding norm equivalence.
We say that the singular integral operators $\mathcal{R}_{j}$, where $0 \leq j \leq d_j$, characterise $\mathsf{H}^1(G)$.

\begin{definition}
Suppose that the singular integral operators $\mathcal{K}^{[i]}_j$, where $0 \leq j \leq n_i$, characterise $\mathsf{H}^1(G_i)$, in the sense above.
The space $\mathsf{H}^1_{\mathrm{SIO}}(\boldsymbol{G})$ is defined to be the set of all $f\in \mathsf{L}^{1}(\boldsymbol{G})$ such that
\[
\sum_{j_1=0}^{n_1}\sum_{j_2=0}^{n_2}\left\Vert  \mathcal{K}_{j_1}^{[1]} \otimes \mathcal{K}_{j_2}^{[2]} f \right\Vert_{\mathsf{L}^{1}(\boldsymbol{G})}<\infty,
\]
with norm
\begin{align*}
\Vert f \Vert_{\mathsf{H}^1_{\mathrm{SIO}}(\boldsymbol{G})}
:=\sum_{j_1=0}^{n_1}\sum_{j_2=0}^{n_2}
\left\Vert  \mathcal{K}_{j_1}^{[1]} \otimes \mathcal{K}_{j_2}^{[2]} f \right\Vert_{\mathsf{L}^{1}(\boldsymbol{G})}.
\end{align*}
\end{definition}

In this section, we generalise Theorem \ref{thm Riesz}, which states that the spaces $\mathsf{H}^1_{\mathrm{Riesz}}(\boldsymbol{G})$ and $\mathsf{H}^1(\boldsymbol{G})$ coincide and have equivalent norms.

\begin{theorem}\label{thm:SIO-Hardy}
Suppose that the singular integral operators $\mathcal{K}^{[i]}_j$, where $0 \leq j \leq n_i$, characterise $\mathsf{H}^1(G_i)$.
Then the double singular integral operators $\mathcal{K}_{j_1}^{[1]} \otimes \mathcal{K}_{j_2}^{[2]}$ characterise the Hardy space $\mathsf{H}^1(\boldsymbol{G})$.
That is, $f\in \mathsf{H}^1(\boldsymbol{G})$ if and only if each $\mathcal{K}_{j_1}^{[1]} \otimes \mathcal{K}_{j_2}^{[2]} f$ is in $\mathsf{L}^{1}(\boldsymbol{G})$ and moreover
\[
\Vert f \Vert_{\mathsf{H}^1(\boldsymbol{G}) }
\eqsim \sum_{j_1=0}^{n_1}\sum_{j_2=0}^{n_2}\left\Vert  \mathcal{K}_{j_1}^{[1]} \otimes \mathcal{K}_{j_2}^{[2]} f \right\Vert_{\mathsf{L}^{1}(\boldsymbol{G})}.
\]
\end{theorem}

It is known (see \cite{HLPW} and \cite{CDLWY}) that singular integral operators associated to homogeneous kernels are bounded from $\mathsf{H}^1(\boldsymbol{G})$ to $\mathsf{L}^1(\boldsymbol{G})$, so it suffices to show that if all the double singular integral transforms of a function $f$ are in $\mathsf{L}^{1}(\boldsymbol{G})$ then $f \in \mathsf{H}^1(\boldsymbol{G})$.
Our proof of Proposition \ref{SIO} below extends \cite{CCLLO},
which introduced a new method, using randomisation, to characterise flag Hardy space on Heisenberg groups by products of singular integrals.

\subsection{The square function and singular integral transforms}

\begin{proposition}\label{SIO}
Suppose that $f\in \mathsf{L}^{2}(\boldsymbol{G})$, and $\mathcal{K}_{j_1}^{[1]} \otimes \mathcal{K}_{j_2}^{[2]} f\in \mathsf{L}^{1}(\boldsymbol{G})$ when $j_i = 0, \dots, n_i $.
Then $f \in \mathsf{H}^1(\boldsymbol{G})$, and
\begin{align*}
\Vert f \Vert_{\mathsf{H}^1(\boldsymbol{G})}\leq C\Vert f \Vert_{\mathsf{H}^1_{\mathrm{SIO}}(\boldsymbol{G})}.
\end{align*}
\end{proposition}
\begin{proof}
We use a randomisation argument coupled with the analogous one-parameter result of Christ and Geller \cite{CG}.
Fix a smooth function $\eta$ on $\mathbb{R}_{+}$, supported in $[1/2,2]$, such that $\sum_{m\in\mathbb{Z}}\eta(2^{-m}t)=1$ for all $t\in \mathbb{R}_{+}$.
By Section \ref{ssec:funct-calc}, the convolution kernels $k_{\eta(\mathcal{L}_i)}$ of the operators $\eta(\mathcal{L}_i)$ on $G_i$ are Schwartz functions of mean $0$.

Let $r_{m}:[0,1]\to \mathbb{R}$ be a collection of independent Rademacher random variables (see \cite{Gra}).
Fix $i$, take $\eta$ as above, and define
\begin{align*}
T_{s}(f) = \sum_{m\in\mathbb{Z}}r_{m}(s)\eta(2^{-m}\mathcal{L}_i )f
\end{align*}
for all $f\in \mathsf{H}^1(G_i )$ and all $s \in [0,1]$.
Straightforward calculation shows that
\begin{align*}
\biggl| \xi^{k}\partial_{\xi}^{k}\biggl( \sum_{m\in\mathbb{Z}}r_{m}(s)\eta(2^{-m}\xi)\biggr) \biggr|
\leq C_{k}
\qquad\forall \xi \in \mathbb{R}_{+} \quad\forall k \in \mathbb{N},
\end{align*}
and from the multiplier theorem (see for example, \cite[Theorem 6.25]{FoSt}), the operator $T_{s}$ is bounded from $\mathsf{H}^1(G_i)$ to $\mathsf{L}^{1}(G_i )$ with norm uniformly bounded for $s\in[0,1]$.
Together with the Christ--Geller characterisation \cite[Theorem A]{CG}, this implies that
\begin{align*}
\biggl\Vert
\sum_{m\in\mathbb{Z}}r_{m}(s)\eta(2^{-m}\mathcal{L}_i )f \biggr\Vert_{\mathsf{L}^{1}(G_i )}
\lesssim \biggl( \sum_{j_i =0}^{n_i }\left\Vert  \mathcal{K}_{j_i }^{[i]}f \right\Vert_{\mathsf{L}^{1}(G_i )}\biggr)
\end{align*}
for all $f\in \mathsf{L}^{1}(G_i )$ such that $\mathcal{K}_{j_i }^{[i]}f\in \mathsf{L}^{1}(G_i )$.
Iteration of the argument shows that
\begin{align*}
&\biggl\Vert  \sum_{m\in\mathbb{Z}}r_{m}(s_1)\eta(2^{-m}\mathcal{L}_1)
 \biggl( \sum_{n\in\mathbb{Z}}r_{n}(s_2)\eta(2^{-n}\mathcal{L}_2)f\biggr)  \biggr\Vert_{\mathsf{L}^{1}(\boldsymbol{G})}\\
&\qquad\lesssim \sum_{j_1=0}^{n_1}
\biggl\Vert  \mathcal{K}_{j_1}^{[1]}\sum_{n\in\mathbb{Z}}r_{n}(s_2)\eta(2^{-n}\mathcal{L}_2)f \biggr\Vert_{\mathsf{L}^{1}(\boldsymbol{G})} \\
&\qquad= \sum_{j_1=0}^{n_1}
\biggl\Vert  \sum_{n\in\mathbb{Z}}r_{n}(s_2)\eta(2^{-n}\mathcal{L}_2)\mathcal{K}_{j_1}^{[1]} f \biggr\Vert_{\mathsf{L}^{1}(\boldsymbol{G})}\\
&\qquad\lesssim \sum_{j_1=0}^{n_1}\sum_{j_2=0}^{n_2}
 \left\Vert  \mathcal{K}_{j_1}^{[1]} \otimes \mathcal{K}_{j_2}^{[2]} f \right\Vert_{\mathsf{L}^{1}(\boldsymbol{G})} ,
\end{align*}
because operators involving convolutions (even with distributions) on $G_1$ and operators involving convolutions (even with distributions) on $G_2$ commute.
By Khinchin's inequality (see, for example, \cite[Appendix C.5]{Gra}), this implies that
\begin{align*}
& \biggl\Vert  \biggl( \sum_{n\in\mathbb{Z}}\sum_{m\in\mathbb{Z}} \left| \eta(2^{-m}\mathcal{L}_1)\eta(2^{-n}\mathcal{L}_2)f \right| ^{2} \biggr) ^{1/2} \biggr\Vert_{\mathsf{L}^{1}(\boldsymbol{G})}\\
&\qquad \lesssim \iint_{[0,1]\times[0,1]}
\biggl\Vert  \sum_{m\in\mathbb{Z}}r_{m}(s_1)\eta(2^{-m}\mathcal{L}_1)
 \sum_{n\in\mathbb{Z}}r_{n}(s_2)\eta(2^{-n}\mathcal{L}_2)f \biggr\Vert_{\mathsf{L}^{1}(\boldsymbol{G})} \,\mathrm{d}s_1\,\mathrm{d}s_2\\
&\qquad\lesssim \sum_{j_1=0}^{n_1}\sum_{j_2=0}^{n_2}\left\Vert  \mathcal{K}_{j_1}^{[1]} \otimes \mathcal{K}_{j_2}^{[2]} f \right\Vert_{\mathsf{L}^{1}(\boldsymbol{G})}.
\end{align*}
This ends the proof of Proposition \ref{SIO}.
\end{proof}

\begin{remark}
It is straightforward to extend this result to products of more than two factors.
It is just a matter of repeating the randomisation argument more times.
\end{remark}

\section{Maximal function characterisation}

In this section, we prove Theorem \ref{thm max}.
As we have already noted, in the one-parameter setting, a common strategy for showing that maximal functions characterise the Hardy space is to use atoms; this strategy does not work in the multi-parameter case.
Merryfield \cite{KM} managed to extend the one-parameter result to the product space $\mathbb{R}^m\times \mathbb{R}^n$; his new tool is the solution of a particular Cauchy--Riemann type equation.
However, it is not clear whether there is a version of his lemma on spaces of homogeneous type, or even just on homogeneous groups.
In \cite{CCLLO}, a new method, using Poisson kernels and harmonic functions, was introduced to characterise flag Hardy space on Heisenberg groups by maximal functions.
Here we extend this method to product groups.

\subsection{The maximal function Hardy spaces}

Recall that $\Gamma^\beta(\boldsymbol{g})$ denotes the cone with vertex $\boldsymbol{g}$ and angle $\beta$:
\[
\Gamma ^{\beta}(\boldsymbol{g})
:= \{(\boldsymbol{h} ,\boldsymbol{t})\in \boldsymbol{G} \times \boldsymbol{T}:
\rho_i (g_i ,h_i ) \leq \beta t_i \text{ when $ i = 1,2$} \}.
\]

\begin{definition}
Take functions $\zeta^{[i]} \in \mathsf{M}(G_i)$, and define the maximal operator $\mathcal{M}_{\zeta,\beta}$ by
\[
\mathcal{M}_{\zeta,\beta}(f)(\boldsymbol{g})
:=\sup_{\boldsymbol{h} \in \Gamma^\beta(\boldsymbol{g})} \left| f \ast \zeta_{\boldsymbol{t}}(\boldsymbol{h}) \right|
\qquad\forall \boldsymbol{g} \in \boldsymbol{G} \quad\forall f \in \mathsf{L}^1(\boldsymbol{G}).
\]
The Hardy space $\mathsf{H}^1_{\max,\zeta,\beta}(\boldsymbol{G})$ is defined to be the space
\[
\{f\in \mathsf{L}^1(\boldsymbol{G}): \left\Vert  \mathcal{M}_{\zeta,\beta} f \right\Vert_{\mathsf{L}^{1}(\boldsymbol{G})}<\infty\}
\]
equipped with the norm
\[
\Vert f \Vert_{\mathsf{H}^1_{\max,\zeta,\beta}(\boldsymbol{G})}
:=\left\Vert  \mathcal{M}_{\zeta,\beta} f \right\Vert_{\mathsf{L}^{1}(\boldsymbol{G})},
\]
for the Hardy space, we require that the integrals of the $\zeta^{[i]}$ are nonzero.
\end{definition}

It is obvious that $\left\Vert  \mathcal{M}_{\zeta,\gamma} f \right\Vert_{\mathsf{L}^1(\boldsymbol{G})}
\leq \left\Vert   \mathcal{M}_{\zeta,\beta} f \right\Vert_{\mathsf{L}^1(\boldsymbol{G})}$ when $\gamma \leq \beta$.

We treat the cases when $\beta > 0$ and when $\beta = 0$ together.
In the important special cases when the $\zeta^{[i]}$ coincide with the Poisson or heat kernels, we have additional tools, such as Harnack or Moser inequalities.
The possibly less well known Plancherel--P\'olya inequality provides similar results for more general $\zeta$.

To characterise $\mathsf{H}^1(\boldsymbol{G})$ by maximal functions, we are going to show two results.

\begin{proposition}\label{prop poisson L-P and max}
If $\beta$ is large enough, then the spaces $\mathsf{H}^1_{\mathrm{sq},P,1}(\boldsymbol{G})$ and $\mathsf{H}^1_{\max,P,\beta}(\boldsymbol{G})$ coincide and have equivalent norms.
\end{proposition}

\begin{proposition}\label{prop equiv max}
For different choices of $\varphi$ and $\zeta$ and different choices of $\beta$ and $\gamma$, the spaces $\mathsf{H}^1_{\max,\varphi,\beta}(\boldsymbol{G})$ and $\mathsf{H}^1_{\max,\zeta,\gamma}(\boldsymbol{G})$ coincide and have equivalent norms.
\end{proposition}

Combining the above two results with Proposition \ref{prop equiv Hardy different beta} proves Theorem \ref{thm max}.

\subsection{Part 1 of the proof of Proposition \ref{prop poisson L-P and max}}

Evidently $\mathsf{H}^1_{\mathrm{sq},P,1}(G) \subseteq\mathsf{H}^1_{\max,P,\beta}(G)$, and
\begin{align}\label{Nf bd by Sf}
\left\Vert  \mathcal{M}_{P,\beta} f \right\Vert_{\mathsf{L}^1(\boldsymbol{G}) }
\lesssim \left\Vert   \mathcal{S}_{P,1} f \right\Vert_{\mathsf{L}^1(\boldsymbol{G}) }.
\end{align}
Indeed, when $\left\Vert   \mathcal{S}_{P,\beta} f \right\Vert_{\mathsf{L}^1(\boldsymbol{G}) }<\infty$, then by \cite{CDLWY}, we may write $f= \sum_j\lambda_j a_j$,
where each $a_j$ is an atom, and $\sum_j\left| \lambda_j \right| \lesssim \left\Vert   \mathcal{S}_{P,\beta} f \right\Vert_{\mathsf{L}^1(\boldsymbol{G}) }$.
Thus, to prove \eqref{Nf bd by Sf}, it suffices to verify that
\[
\left\Vert  \mathcal{M}_{P,\beta} (a) \right\Vert_{\mathsf{L}^1(\boldsymbol{G}) }\lesssim1
\]
for each product atom $a$ as in Section 3.1.
The Poisson kernels $P^{[1]}_1$ and $P^{[2]}_1$ satisfy the standard decay and smoothness conditions \eqref{eq:molecule}, and the atom $a$ satisfies the standard product cancellation condition \eqref{eq:product-atom-cancel}.
Then the desired estimate of $\left\Vert  \mathcal{M}_{P,\beta} (a) \right\Vert_{\mathsf{L}^1(\boldsymbol{G}) }$ follows from standard product arguments and Journ\'e's covering lemma.

We provide a brief outline of the proof here for completeness and for the reader's convenience.
It suffices to show that, for each product atom $a$,
\[
\left\Vert  \mathcal{M}_{P,\beta} (a) \right\Vert_{\mathsf{L}^1(\boldsymbol{G}) }\lesssim1 .
\]
From Section 3.1, we may write $a$ on $\boldsymbol{G} $ as a sum $\sum_{R\in \mathscr{M}(U)}a_R$, where $U$ is an open subset of $\boldsymbol{G}$ of finite measure, and the particles $a_R$ satisfy support and size conditions.

There are two steps to the proof:  first, we find some small positive $\epsilon$, such that for any pseudodyadic rectangle $S$ containing the pseudodyadic rectangle $R$,
\begin{equation}\label{eq:good-estimate}
\int_{\boldsymbol{G} \setminus S} | \mathcal{M}_{P,\beta} (a_R)(\boldsymbol{g})| \,\mathrm{d}\boldsymbol{g}
\lesssim
\left( \frac{\ell_1(R)}{\ell_1(S)} + \frac{\ell_2(R)}{\ell_2(S)}\right)^\epsilon |R|^{1/2} \left\Vert a_R\right\Vert_{\mathsf{L}^ 2(\boldsymbol{G})} ,
\end{equation}
for all particles $a_R$ associated to $R$.
The estimation of this expression may be achieved by writing $S$ as $Q_1 \times Q_2$ and breaking up $\boldsymbol{G}\setminus S$ into the three regions $(Q_1)^c \times (Q_2)^c$, $(Q_1)^c \times Q_2$ and $Q_1 \times (Q_2)^c$.
In the first region we use the support and cancellation conditions on $a_R$ to estimate $a_R * P_t(\boldsymbol{g})$, and show that
\[
\int_{(Q_1)^c \times (Q_2)^c} | \mathcal{M}_{P,\beta} (a_R)(\boldsymbol{g})| \,\mathrm{d}\boldsymbol{g}
\lesssim
\left( \frac{\ell_1(R)}{\ell_1(S)} \cdot \frac{\ell_2(R)}{\ell_2(S)} \right)^\epsilon |R|^{1/2} \left\Vert a_R\right\Vert_{\mathsf{L}^ 2(\boldsymbol{G})}
;
\]
in the second region we use the support and cancellation conditions to control the $g_1$ variable and H\"older's inequality and a Littlewood--Paley argument to control the $g_2$ variable, and show that
\[
\int_{(Q_1)^c \times Q_2} | \mathcal{M}_{P,\beta} (a_R)(\boldsymbol{g})| \,\mathrm{d}\boldsymbol{g}
\lesssim
\left( \frac{\ell_1(R)}{\ell_1(S)} \right)^\epsilon |R|^{1/2} \left\Vert a_R\right\Vert_{\mathsf{L}^ 2(\boldsymbol{G})}
;
\]
in the third region, we argue similarly, but with the roles of the variables reversed.

Once \eqref{eq:good-estimate} is proved, we apply Journ\'e's lemma to obtain an estimate for atoms rather than particles.
Let $a$ be an atom associated to the open set $U$, and write $a = \sum_{R \in m(U)} a_R$.
Define
\[
V = \left\{ \boldsymbol{g} \in \boldsymbol{G} : \mathcal{M}_{S} \chi_{U}(\boldsymbol{g}) > {1}/{4} \right\}
\qquad\text{and}\qquad
W = \left\{ \boldsymbol{g} \in \boldsymbol{G} : \mathcal{M}_{S} \chi_{V}(\boldsymbol{g}) > {1}/{4} \right\}.
\]
Then $|W| \lesssim |V| \lesssim |U|$.
Further,
\begin{align*}
&\int_{\boldsymbol{G}} | \mathcal{M}_{P,\beta} (a)(\boldsymbol{g})| \,\mathrm{d}\boldsymbol{g} \\
&\qquad= \int_{W} | \mathcal{M}_{P,\beta} (a)(\boldsymbol{g})| \,\mathrm{d}\boldsymbol{g}
+ \int_{\boldsymbol{G} \setminus W} | \mathcal{M}_{P,\beta} (a)(\boldsymbol{g})| \,\mathrm{d}\boldsymbol{g} \\
&\qquad\leq |W|^{1/2} \left( \int_{W} |\mathcal{M}_{P,\beta} (a)(\boldsymbol{g})| \,\mathrm{d}\boldsymbol{g} \right)^{1/2}
+ \sum_{R \in m(U)} \int_{\boldsymbol{G} \setminus W} | \mathcal{M}_{P,\beta} (a_R)(\boldsymbol{g})| \,\mathrm{d}\boldsymbol{g}.
\end{align*}
The first term is estimated using the $L^2$-boundedness of $\mathcal{M}_{P,\beta}$:
\[
|W|^{1/2} \left( \int_{W} |\mathcal{M}_{P,\beta} (a)(\boldsymbol{g})| \,\mathrm{d}\boldsymbol{g} \right)^{1/2}
\lesssim |W|^{1/2} \left\Vert a\right\Vert_{\mathsf{L}^ 2(\boldsymbol{G})}
\lesssim |U|^{1/2} |U|^{-1/2} = 1.
\]
The second term is estimates using Journ\'e's lemma, namely, for each $R \in m(U)$, we can find $S \in m(W)$ (depending on $R$) such that $R \subseteq S$ and
\[
\sum_{r \in m(R)} \left( \frac{\ell_1(R)}{\ell_1(S)} + \frac{\ell_2(R)}{\ell_2(S)}\right)^\epsilon |R| \lesssim |U|.
\]
Then from \eqref{eq:good-estimate}, H\"older's inequality, Journ\'e's lemma, and the definitions,
\begin{align*}
&\sum_{R \in m(U)} \int_{\boldsymbol{G} \setminus W} | \mathcal{M}_{P,\beta} (a_R)(\boldsymbol{g})| \,\mathrm{d}\boldsymbol{g} \\
&\qquad\leq \sum_{R \in m(U)} \int_{\boldsymbol{G} \setminus S} | \mathcal{M}_{P,\beta} (a_R)(\boldsymbol{g})| \,\mathrm{d}\boldsymbol{g} \\
&\qquad\lesssim \sum_{R \in m(U)} \left( \frac{\ell_1(R)}{\ell_1(S)} + \frac{\ell_2(R)}{\ell_2(S)}\right)^\epsilon |R|^{1/2} \left\Vert a_R\right\Vert_{\mathsf{L}^ 2(\boldsymbol{G})} \\
&\qquad\lesssim \left( \sum_{R \in m(U)} \left( \frac{\ell_1(R)}{\ell_1(S)} + \frac{\ell_2(R)}{\ell_2(S)}\right)^{2\epsilon} |R|\right) ^{1/2}
\left( \sum_{R \in m(U)}  \left\Vert a_R\right\Vert_{\mathsf{L}^ 2(\boldsymbol{G})}^2 \right) ^{1/2} \\
&\qquad\lesssim |U|^{1/2} |U|^{-1/2} = 1.
\end{align*}

Note that proving this result of products of more than two factors seems to be nontrivial; Journ\'e's lemma requires us to have more than one ``improving factor'' $\ell_j(R)/\ell_j(S)$.

It remains to prove the opposite inclusion: $\mathsf{H}^1_{\max,P,\beta}(\boldsymbol{G}) \subseteq \mathsf{H}^1_{\mathrm{sq},P,1}(\boldsymbol{G})$, and
\begin{align*}
\left\Vert   \mathcal{S}_{P,1} f \right\Vert_{\mathsf{L}^1(\boldsymbol{G}) }
\lesssim \left\Vert  \mathcal{M}_{P,\beta} f \right\Vert_{\mathsf{L}^1(\boldsymbol{G}) }.
\end{align*}

We first treat a stratified group, and then a product of stratified groups.

\subsection{Part 2 of the proof of Proposition \ref{prop poisson L-P and max}}
\label{ssec:LP-and-max}
In this part of the proof, we prove \eqref{L-P bounded by max-1} for a stratified group ${G}$ with no product structure.
This simplifies the notation.
Later the group ${G}$ will be one of the factors of the product group $\boldsymbol{G}$ that we wish to consider.

We are going to use integration by parts, and need to know about the behaviour of certain harmonic functions on ${G} \times \mathbb{R}_{+}$ at the boundaries of this region.
Suppose that $f \in \mathsf{L}^{p}({G})$, where $1 \leq p \leq \infty$, and consider the Poisson integral $f \ast P_{t}(g)$ and $f \ast Q_{t}(g)$, where $g \in {G}$ and $t \in \mathbb{R}_{+}$, whose behaviour as $t \to 0$ is discussed in Corollary \ref{cor:Lebesgue}.

From Lemma \ref{lem:heat-pois-estim}, if $f \in \mathsf{L}^\infty({G})$, then $f \ast P_{t}$ and $f \ast Q_{t}$ are bounded in $\mathsf{L}^\infty({G})$ as $t \to \infty$.
If $f \in \mathsf{H}^1_{\max,P,\gamma}({G})$, then $\left\Vert  f_{1/s} \ast P_{1}\right\Vert_{\mathsf{H}^1_{\max,P,\gamma}({G})}$ is bounded for all $s>0$.
Since
\[
f_{1/s} \ast P_{1} \to \left( \int_{{G}} f(g) \,\mathrm{d}g\right)  P_{1}
\qquad\text{as $s \to \infty$}
\]
in $\mathsf{L}^1({G})$ and $\sup _{t > 1} P_{t}(\cdot) \notin \mathsf{L}^1(G)$ so $P_{1} \notin \mathsf{H}^1_{\max,P,\gamma}({G})$, we see that $f$ has mean $0$.
Thus
\[
\begin{aligned}
&\left\Vert  f \ast P_{t} \right\Vert_{\mathsf{L}^1({G})}
+ \left\Vert  f \ast Q_{t} \right\Vert_{\mathsf{L}^1({G})} \\
&\qquad= \left\Vert  f_{1/t} \ast P_{1} \right\Vert_{\mathsf{L}^1({G})}
+ \left\Vert  f_{1/t} \ast Q_{1} \right\Vert_{\mathsf{L}^1({G})}
 \to 0
\end{aligned}
\]
and
\[
\begin{aligned}
&\left\Vert  f \ast P_{t} \right\Vert_{\mathsf{L}^\infty({G})}
+ \left\Vert  f \ast Q_{t} \right\Vert_{\mathsf{L}^\infty({G})} \\
&\qquad= t^{-Q} \left\Vert  f_{1/t} \ast P_{1} \right\Vert_{\mathsf{L}^\infty({G})}
+ t^{-Q} \left\Vert  f_{1/t} \ast Q_{1} \right\Vert_{\mathsf{L}^\infty({G})}
 \to 0
\end{aligned}
\]
as $t \to \infty$.
This convergence is also pointwise almost everywhere.

\begin{proposition}
Suppose that ${G}$ is a stratified Lie group and $\gamma > 0$.
If $\beta$ is large enough, then
\[
\mathsf{H}^1_{\max,P,\beta}({G})
\subseteq \mathsf{H}^1_{\mathrm{sq},P,\gamma}({G}),
\]
and there is a corresponding norm inequality:
\begin{equation}\label{L-P bounded by max-1}
\Vert f \Vert_{\mathsf{H}^1_{\mathrm{sq},P,\gamma}({G})}
\lesssim \Vert f \Vert_{\mathsf{H}^1_{\max,P,\beta}({G})}
\qquad\forall f \in \mathsf{H}^1_{\max,P,\beta}({G}).
\end{equation}
\end{proposition}

\begin{proof}
Take $f\in \mathsf{L}^{1}({{G}})$ such that $\mathcal{M}_{P,\beta}(f)\in \mathsf{L}^{1}({{G}})$.
We assume that $f$ is real-valued, for otherwise we may treat the real and imaginary parts separately.
We may also suppose that $f$ is smooth, by a simple mollification argument.

Fix $\alpha > 0$, and define
\begin{gather*}
L_{\beta}(\alpha)
:=\left\{ g \in {{G}}: \mathcal{M}_{P,\beta}(f)(g) \leq \alpha \right\}  ,
\\
A_{\beta}(\alpha)
:=\left\{ g \in {{G}}:
\mathcal{M}_{S}(1 - \chi_{L_{\beta}(\alpha)})(g)<\frac{1}{4} \right\} ,
\end{gather*}
where $\mathcal{M}_{S}$ is the strong maximal operator, which is $\mathsf{L}^{2}$ bounded.
Then
\begin{align}\label{eq:size-A-L-1}
A_{\beta}(\alpha) \subseteq L_{\beta}(\alpha)
\qquad\text{and}\qquad
\left| (L_{\beta}(\alpha))^c \right|
\leq \left| A_{\beta}(\alpha)^{c} \right|
\lesssim \left| (L_{\beta}(\alpha))^c \right| .
\end{align}
Define also
\[
W_{\beta}:=\bigcup_{g\in A_{\beta}(\alpha)}\Gamma^{\beta}(g)
\qquad\text{and}\qquad
\widetilde{W}_{\beta}:=\bigcup_{h\in L_{\beta}(\alpha)(f)}\Gamma^{\beta}(h).
\]

We claim that there exists $C_{0}\in(0,1)$ such that
\begin{align*}%\label{C-zero-1}
\chi_{L_\gamma(\alpha)} \ast P_{t}(g) \geq C_{0}
\qquad\forall (g,t)\in W_\gamma.
\end{align*}
Indeed, by definition, for such $(g,t)$,
\begin{align*}
(1 - \chi_{L_\gamma(\alpha)}) \ast \chi_{B(o,\gamma t)}
<\frac{1}{4} \,\left| B(o,\gamma t) \right|  ,
\end{align*}
that is,
\begin{align*}
\chi_{L_\gamma(\alpha)} \ast \chi_{B(o,\gamma t)} \geq \frac{3}{4} \,\left| B(o,\gamma t) \right|  ,
\end{align*}
and the claim follows from Lemma \ref{lem:heat-pois-estim}.
We also claim that if $\beta$ is large enough, then there is a constant $C_1\in(0,C_{0})$, such that if $(g,t)\notin \widetilde{W}_{\beta}$, then
\begin{align*}%\label{C-one-1}
\chi_{L_{\beta}(\alpha)} \ast P_{t}(g)\leq C_1.
\end{align*}
Indeed, if $(g,t)\notin \widetilde{W}_{\beta}$ then $\rho (h ^{-1}g)\geq \beta t $ for all $h \in L_{\beta}(\alpha)$.
Hence,
\begin{align*}
\chi_{L_{\beta}(\alpha)} \ast P_{t}(g)
&=\int_{G}
\chi_{L_{\beta}(\alpha)}(h)
P_{t}(h^{-1}g) \,\mathrm{d}h
%\\&
\leq \int_{B(g, \beta t)^c} P_{t}(h^{-1}g)\,\mathrm{d}h\\
&= \int_{B(g, \beta t)^c} P_{1}(h^{-1}g)\,\mathrm{d}h
%\\&
\to 0
\end{align*}
as $\beta\to \infty$, proving our claim.

Take a smooth function $\eta: \mathbb{R}\to \mathbb{R}$ such that $\eta(s)=1$ when $s\geq C_{0}$ and $\eta(s)=0$ when $s\leq C_1$.
Define $\mathrm{H}_{t} := \chi_{L_{\beta}(\alpha)} \ast P_{t}$.
Then
\[
t \partial_{t} \mathrm{H}_{t}(g)
= \chi_{L_{\beta}(\alpha)} \ast Q_{t}(g),
\]
which is uniformly bounded for all $g \in {G}$ and $t \in \mathbb{R}_{+}$ and
\[
t \partial_{t} \mathrm{H}_{t} \to 0
\qquad\text{as $t \to 0$}
\]
pointwise almost everywhere, by Corollary \ref{cor:Lebesgue},

It will suffice to show that
\begin{equation}\label{eq:need-to-show-1p-1}
\begin{aligned}
&\int_{A_{\gamma}(\alpha)} \mathcal{S}_{P,\gamma}(f)(g)^{2} \,\mathrm{d}g
%\\&\qquad
\lesssim \int_{L_{\beta}(\alpha)} \mathcal{M}_{P,\beta}(f)(g)^{2}\,\mathrm{d}g
 + \alpha^{2}\left| L_{\gamma}(\alpha)^c \right| .
\end{aligned}
\end{equation}
Indeed, coupled with \eqref{eq:size-A-L-1}, this implies that
\begin{align*}
&\left| \{ g \in {G} : \mathcal{S}_{P,\gamma}(f)(g) > \alpha\} \right|  \\
&\qquad\leq \left| \{g \in A_{\gamma}(\alpha)^{c}: \mathcal{S}_{P,\gamma}(f)(g)>\alpha\} \right|
+ \left| \{g\in A_{\gamma}(\alpha): \mathcal{S}_{P,\gamma}(f)(g)>\alpha\} \right|  \\
&\qquad\leq \left| A_{\gamma}(\alpha)^c \right|
+ \frac{1}{\alpha^{2}}\int_{A_{\gamma}(\alpha)} \mathcal{S}_{P,\gamma}(f)(g)^{2} \,\mathrm{d}g \\
&\qquad\lesssim \left| L_{\gamma}(\alpha)^c \right|
+ \frac{1}{\alpha^{2}}
\int_{L_{\beta}(\alpha)} \mathcal{M}_{P,\beta}(f)(g)^{2}\,\mathrm{d}g.
\end{align*}
A standard integration with respect to $\alpha$ then implies that
\begin{align*}
\left\Vert  \mathcal{S}_{P,\gamma}(f) \right\Vert_{\mathsf{L}^{1}({G})}
\lesssim \left\Vert  \mathcal{M}_{P,\beta}(f) \right\Vert_{\mathsf{L}^{1}({G})},
\end{align*}
that is, the required estimate \eqref{L-P bounded by max-1} holds.

We observe that
\begin{equation*}%\label{a0}
\begin{aligned}
\int_{A_{\beta}(\alpha)} \mathcal{S}_{P,\gamma}(f)(g)^{2}\,\mathrm{d}g
&= \int_{A_{\beta}(\alpha)} \iint_{\Gamma^{\gamma}(g)}
\left|  \slashed{\nabla}(f \ast P_{t}) (h) \right| ^{2}
 \frac{t}{\left| B(o,t) \right| } \,\mathrm{d}t \,\mathrm{d}h \,\mathrm{d}g\\
&\lesssim \iint_{W_{\gamma}}
\left|  \slashed{\nabla}(f \ast P_{t})(g) \right| ^{2} t \,\mathrm{d}t\,\mathrm{d}g\\
&\leq\iint_{{G} \times \mathbb{R}_{+}} \left| \slashed{\nabla}(f \ast P_{t})(g) \right| ^{2}
%\\&\hspace{5.2cm} \times
\left| \eta( \mathrm{H}_{t}(g))\right| ^{2} t \,\mathrm{d}t\,\mathrm{d}g.
\end{aligned}
\end{equation*}
From \eqref{eq:need-to-show-1p-1}, it will therefore suffice to show that
\begin{equation}\label{eq:need-to-show-1p-2}
\begin{aligned}
\mathrm{I}_0
&:=\iint_{{G} \times \mathbb{R}_{+} } \left| \slashed{\nabla}(f \ast P_{t})(g) \right| ^{2}
%\\&\hspace{5.2cm} \times
\left| \eta( \mathrm{H}_{t}(g))\right| ^{2} t \,\mathrm{d}t\,\mathrm{d}g
\\&\lesssim \int_{L_{\beta}(\alpha)} \mathcal{M}_{P,\beta}(f)(g)^{2}\,\mathrm{d}g
 + \alpha^{2}\left| L_{\gamma}(\alpha)^c \right| .
\end{aligned}
\end{equation}

We note that $u : (g,t) \mapsto F \ast P_{t}(g)$ is harmonic on ${G} \times \mathbb{R}_{+}$ for all $F\in \mathsf{L}^{1}({G}) + \mathsf{L}^{\infty}({G})$, in the sense that
\begin{align*}
\slashed{\mathcal{L}} u (g ,t )=0,
\end{align*}
where $\slashed{\mathcal{L}} := \mathcal{L} - \partial_{t}^{2}$.
Consequently,
\begin{align*}
\left| \slashed{\nabla} u (g ,t) \right| ^{2}
= - \frac{1}{2}\slashed{\mathcal{L}} \left( u ^{2}(g,t)\right)
\qquad\forall (g ,t )\in {G} \times \mathbb{R}_{+}.
\end{align*}
Further, by our remark on harmonicity, $\slashed{\mathcal{L}} \mathrm{H}_{t} = 0$, and so
\[
\begin{aligned}
\slashed{\mathcal{L}} \eta( \mathrm{H}_{t}(g))
&= \slashed{\nabla} \cdot (\eta'( \mathrm{H}_{t}(g)) \slashed{\nabla} \mathrm{H}_{t}(g)) \\
&= \eta''( \mathrm{H}_{t}(g)) \left| \slashed{\nabla} \mathrm{H}_{t}(g)) \right| ^2.
\end{aligned}
\]
It follows that
\begin{equation}\label{eq:harmonicity-1p}
\begin{aligned}
\left| \slashed{\nabla} (f \ast P_{t}) (g) \eta( \mathrm{H}_{t}(g))\right| ^{2}
&
= -\frac{1}{2}\slashed{\mathcal{L}} \left(  \left|  f \ast P_{t} (g) \eta( \mathrm{H}_{t}(g))\right| ^{2}\right)
\\&\qquad
-4 f\ast P_{t}(g)\eta( \mathrm{H}_{t}(g))
\slashed{\nabla} (f\ast P_{t})(g) \cdot \slashed{\nabla} \eta(\mathrm{H}_{t}(g))\\
&\qquad-\left| f\ast P_{t}(g) \right| ^2 \left| \slashed{\nabla} \eta(\mathrm{H}_{t}(g))\right| ^{2}
\\&\qquad
-\left| f\ast P_{t}(g) \right| ^2 \eta( \mathrm{H}_{t}(g)) \eta''( \mathrm{H}_{t}(g)) \left| \slashed{\nabla} \mathrm{H}_{t}(g)) \right| ^2 .
\end{aligned}
\end{equation}
We estimate the second, third, and fourth terms on the right hand side of \eqref{eq:harmonicity-1p} as follows.
First, by the arithmetic--geometric mean inequality and the chain rule,
\[
\begin{aligned}
&\left|  4 f\ast P_{t}(g) \eta(\mathrm{H}_{t}(g))
\slashed{\nabla} (f\ast P_{t})(g) \cdot \slashed{\nabla} \eta(\mathrm{H}_{t}(g)) \right|
\\
&\qquad\leq \frac{1}{2} \left| \slashed{\nabla} P_{t}\ast f(g) \right| ^{2}\left| \eta( \mathrm{H}_{t}(g)) \right| ^{2}
+ 8 \left|  f \ast P_{t} (g) \right| ^{2} \left|  \slashed{\nabla} \eta(\mathrm{H}_{t}(g)) \right| ^2 \\
&\qquad\leq \frac{1}{2} \left| f \ast \slashed{\nabla} P_{t}(g) \right| ^{2}\left| \eta(\mathrm{H}_{t}(g)) \right| ^{2}
+ 8 \left\Vert \eta'\right\Vert_{\mathsf{L}^\infty(\mathbb{R})} \left|  f \ast P_{t} (g) \right| ^{2}
\left| \slashed{\nabla} \mathrm{H}_{t}(g) \right| ^2
\end{aligned}
\]
and we can move the first term on the right hand side of this inequality to the left hand side of \eqref{eq:harmonicity-1p}.
Next,
\[
\left| f\ast P_{t}(g) \right| ^2 \left| \slashed{\nabla} \eta(\mathrm{H}_{t}(g))\right| ^{2}
\leq \left\Vert \eta'\right\Vert_{\mathsf{L}^\infty(\mathbb{R})}^2 \left| f\ast P_{t}(g) \right| ^2 \left| \slashed{\nabla} \mathrm{H}_{t}(g)\right| ^{2}
\]
and similarly,
\[
\begin{aligned}
&\left| f\ast P_{t}(g) \right| ^2 \left| \eta( \mathrm{H}_{t}(g)) \right|  \left| \eta''( \mathrm{H}_{t}(g)) \right|  \left| \slashed{\nabla} \mathrm{H}_{t}(g)) \right| ^2 \\
&\qquad\leq \left\Vert \eta\right\Vert_{\mathsf{L}^\infty(\mathbb{R})} \left\Vert \eta''\right\Vert_{\mathsf{L}^\infty(\mathbb{R})}
 \left| f\ast P_{t}(g) \right| ^2 \left| \slashed{\nabla} \mathrm{H}_{t}(g)) \right| ^2.
\end{aligned}
\]
We conclude that
\begin{equation}\label{eq:harmonicity-1p-bis}
\begin{aligned}
&\left| \slashed{\nabla} (f \ast P_{t}) (g) \eta( \mathrm{H}_{t}(g))\right| ^{2}\\
&\qquad
\leq - \slashed{\mathcal{L}} \left(  \left|  f \ast P_{t} (g) \eta(\mathrm{H}_{t}(g)) \right| ^{2}\right)
%\\&\qquad\qquad
+ C(\eta) \left| f\ast P_{t}(g) \right| ^2 \left| \slashed{\nabla} \mathrm{H}_{t}(g)) \right| ^2 \\
&\qquad=:f_1(g,t)+f_2(g,t),
\end{aligned}
\end{equation}
say.
The proof of \eqref{eq:need-to-show-1p-2} is now straightforward.

Evidently, $\mathrm{I}_0 \leq \mathrm{I}_1 + \mathrm{I}_2$, where
\begin{align*}
{\mathrm{I}}_j
= \left|  \iint_{{G} \times \mathbb{R}_{+}} f_j (g,t)t \,\mathrm{d}t\,\mathrm{d}g \right| .
\end{align*}

To treat the term ${\mathrm{I}_1}$, we recall that $\slashed{\mathcal{L}} = \mathcal{L} - \partial_{t}^2$.
Integration by parts and the decay of the Poisson integral $f\ast P_{t}$ at infinity imply that
\begin{align*}
&\int_{{G}} \mathcal{L}
\Bigl( \left| f\ast P_{t}(g)
\eta( \mathrm{H}_{t}(g))\right| ^{2}
\Bigr) t \,\mathrm{d}g = 0
\end{align*}
for all $t > 0$, and also that
\begin{align*}
&\int_{\mathbb{R}_{+}}\partial_{t}^2
\Bigl( \left| f\ast P_{t}(g)
\eta( \mathrm{H}_{t}(g))\right| ^{2}
\Bigr) t\,\mathrm{d}t\\
&\qquad= \left[ t \partial_{t}
\left( \left|  f\ast P_{t}(g)
\eta( \mathrm{H}_{t}(g))\right| ^{2}\right) \right] _{t=0}^{t=\infty}
%\\&\qquad\qquad
-\int_{\mathbb{R}_{+}}\partial_{t}
\left( \left|  f\ast P_{t}(g)
\eta( \mathrm{H}_{t}(g))\right| ^{2} \right)  \,\mathrm{d}t \\
&\qquad= 2\Bigl[
\left(  f\ast P_{t}(g)
\eta( \mathrm{H}_{t}(g))\right)
%\\&\qquad\qquad\qquad\qquad \times
t \partial_{t} \left(  f\ast P_{t}(g)
\eta( \mathrm{H}_{t}(g))\right) \Bigr]_{t=0}^{t=\infty}
%\\&\qquad\qquad
- \Bigl[ \left|  f\ast P_{t}(g)
\eta( \mathrm{H}_{t}(g))\right| ^{2} \Bigr]_{t=0}^{t=\infty} \\
&\qquad=
\left|  f(g) \chi_{L_\beta(\alpha)}(g) \right| ^2;
\end{align*}
many of the terms here when $t=0$ vanish by our remarks before the enunciation of this proposition.
Therefore
\begin{align*}
{\mathrm{I}_1}
&=\int_{L_\beta(\alpha)} \left| f(g)\right| ^{2} \,\mathrm{d}g ,
\end{align*}
which is the first term on the right hand side of \eqref{eq:need-to-show-1p-2}.

Next, since $\left|  f \ast P_{t} (g) \right|  \leq \mathcal{M}_{P,\beta}(g) \leq \alpha$ when $(g,t) \in W_\beta$, and $\slashed{\nabla} P_{t}$ has mean $0$,
\begin{align*}
\mathrm{I}_2
&=\iint_{W_\beta} \left| f\ast P_{t}(g) \right| ^2
\left| \slashed{\nabla} \mathrm{H}_{t}(g)\right| ^{2} t \,\mathrm{d}t\,\mathrm{d}g \\
&\leq \alpha^2 \iint_{W_\beta}
\left| \slashed{\nabla} \mathrm{H}_{t}(g)\right| ^{2} t \,\mathrm{d}t\,\mathrm{d}g \\
&\leq \alpha^2 \iint_{{G} \times \mathbb{R}_{+}} \left| \chi_{L_\beta(\alpha)} \ast t\slashed{\nabla} P_{t} (g))  \right| ^2 \,\frac{\mathrm{d}t}{t} \,\mathrm{d}g \\
&= \alpha^2 \iint_{{G} \times \mathbb{R}_{+}} \left| (1 - \chi_{L_\beta(\alpha)}) \ast t\slashed{\nabla} P_{t} (g))  \right| ^2 \,\frac{\mathrm{d}t}{t}\,\mathrm{d}g \\
&\eqsim \alpha^2 \left|  L_\beta(\alpha)^c  \right| ^2,
\end{align*}
by Littlewood--Paley theory.
This is the second term on the right hand side of \eqref{eq:need-to-show-1p-2}, and the proposition is now proved.
\end{proof}

\begin{remark}\label{rem:integ-by-parts}
We summarise the first step of this proof as the application of harmonicity to estimate the desired square function as a sum of two terms in \eqref{eq:harmonicity-1p-bis}.
The ``main term'' $\mathrm{I}_1$ gives us the function $f$ that we started with, while the ``error term'' $\mathrm{I}_2$ gives us an expression that we can handle by using Littlewood--Paley arguments.
\end{remark}

\subsection{Part 3 of the proof of Proposition \ref{prop poisson L-P and max}}

It remains to take a product group $\boldsymbol{G}$, prove the inclusion and inequality
\begin{gather*}%\label{L-P bounded by max-2p}
\mathsf{H}^1_{\max,P,\beta}(\boldsymbol{G})
\subseteq \mathsf{H}^1_{\mathrm{sq},P,1}(\boldsymbol{G}) \\
\left\Vert   \mathcal{S}_{P,1} f \right\Vert_{\mathsf{L}^1(\boldsymbol{G}) }
\lesssim \left\Vert  \mathcal{M}_{P,\beta} f \right\Vert_{\mathsf{L}^1(\boldsymbol{G}) }
\qquad\forall f \in \mathsf{H}^1_{\max,P,\beta}(\boldsymbol{G}).
\end{gather*}
Again we may and do assume that $f$ is real-valued and smooth.

The initial definitions are the same as in the one-parameter case.
Take $f\in \mathsf{L}^{1}(\boldsymbol{G})$ such that $\mathcal{M}_{P,\beta}(f)\in \mathsf{L}^{1}(\boldsymbol{G})$ and $\alpha>0$.
Define
\begin{gather*}
L_{\beta}(\alpha)
:=\left\{ \boldsymbol{g} \in \boldsymbol{G}: \mathcal{M}_{P,\beta}(f)(\boldsymbol{g}) \leq \alpha \right\}  ,
\\
A_{\beta}(\alpha)
:=\left\{ \boldsymbol{g} \in \boldsymbol{G}:
\mathcal{M}_{S}(1 - \chi_{L_{\beta}(\alpha)})(\boldsymbol{g})<\frac{1}{4} \right\} ,
\end{gather*}
where $\mathcal{M}_{S}$ denotes the strong maximal operator.
By the same argument as in the one-parameter case,
\begin{align*}%\label{eq:size-A-L1-2p}
A_{\beta}(\alpha)
\subseteq
L_{\beta}(\alpha)
\qquad\text{and}\qquad
\left| A_{\beta}(\alpha)^{c} \right| \lesssim \left| L_{\beta}(\alpha)^c \right| .
\end{align*}
Define also
\[
W_{\beta}:=\bigcup_{\boldsymbol{g}\in A_{\beta}(\alpha)}\Gamma^{\beta}(\boldsymbol{g})
\qquad\text{and}\qquad
\widetilde{W}_{\beta}:=\bigcup_{\boldsymbol{h}\in L_{\beta}(\alpha)(f)}\Gamma^{\beta}(\boldsymbol{h}).
\]

As in the one-parameter case, there exists $C_{0}\in(0,1)$ and $C_1\in(0,C_{0})$, such that
\begin{align*}%\label{C-zero-C-one}
\chi_{L_1(\alpha)} \ast P_{\boldsymbol{t}}(\boldsymbol{g}) &\geq C_{0}
\qquad\forall (\boldsymbol{g},\boldsymbol{t})\in W_\gamma \\
\chi_{L_{\beta}(\alpha)} \ast P_{\boldsymbol{t}}(\boldsymbol{g}) &\leq C_1
\qquad\forall (\boldsymbol{g},\boldsymbol{t})\in (\widetilde{W}_{\beta})^c ,
\end{align*}
provided that $\beta$ is large enough.

We let $\mathrm{H}_{\boldsymbol{t}} := \chi_{L_{\beta}(\alpha)} \ast P_{\boldsymbol{t}}$; here $t_1,t_2 \geq 0$.
Take a smooth real-valued function $\eta$ on $\mathbb{R}$ such that $\eta(s)=1$ when $s\geq C_{0}$ and $\eta(s)=0$ when $s\leq C_1$.
By definition,
\[
t_1 \partial_{t_1} \mathrm{H}_{\boldsymbol{t}}(\boldsymbol{g})
= \chi_{L_{\beta}(\alpha)} \ast (Q_{t_1} \otimes P_{t_2})(\boldsymbol{g}),
\]
and this is uniformly bounded for all $\boldsymbol{g} \in \boldsymbol{G}$ and $t_1, t_2 \geq 0$; further, by \eqref{eq:chi-star-Q},
\[
t_1 \partial_{t_1} \mathrm{H}_{\boldsymbol{t}}(\boldsymbol{g}) \to 0
\qquad\text{as $t_1 \to 0$}
\]
for almost all $\boldsymbol{g} \in \boldsymbol{G}$.

Again, it will suffice to show that
\begin{equation*}%\label{eq:need-to-show-1}
\begin{aligned}
&\iint_{\boldsymbol{G} \times \boldsymbol{T}} \left| \slashed{\nabla}_1\slashed{\nabla}_2(f \ast P_{\boldsymbol{t}})(\boldsymbol{g}) \right| ^{2}
%\\&\hspace{5.2cm} \times
\left| \eta( \mathrm{H}_{\boldsymbol{t}}(\boldsymbol{g}))\right| ^{2} \boldsymbol{t} \,\mathrm{d}\boldsymbol{t} \,\mathrm{d}\boldsymbol{g}
\\&\qquad
\lesssim \int_{L_{\beta}(\alpha)} \mathcal{M}_{P,\beta}(f)(\boldsymbol{g})^{2}\,\mathrm{d}\boldsymbol{g}
 + \alpha^{2}\left| L_{1}(\alpha)^c \right| .
\end{aligned}
\end{equation*}
We do this by extending the computation for a single homogeneous group.

First, we fix the variables $g_2$ and $t_2$.
By the one-parameter case,
\[
\left\Vert \mathcal{S}_{P,\gamma}^{[1]}(f)(\cdot,g_2)\right\Vert_{\mathsf{L}^1(G_1)}
\lesssim
\left\Vert \mathcal{M}_{P,\beta}^{[1]} (f)(\cdot, g_2)\right\Vert_{\mathsf{L}^1(G_1)},
\]
whence
\begin{equation}\label{eq:2-parameter-case}
\left\Vert  \mathcal{S}_{P,\gamma}^{[1]}(f)\right\Vert_{\mathsf{L}^1(\boldsymbol{G})}
\lesssim
\left\Vert \mathcal{M}_{P,\beta}^{[1]} (f)\right\Vert_{\mathsf{L}^1(\boldsymbol{G})}
\leq
\left\Vert \mathcal{M}_{P,\beta} (f)\right\Vert_{\mathsf{L}^1(\boldsymbol{G})}
\end{equation}
by integration over $G_2$ and the pointwise inequality $\mathcal{M}_{P,\beta}^{[1]} (f) \leq \mathcal{M}_{P,\beta} (f)$.
A similar result holds for the Littlewood--Paley operator acting in the second variable only.

The function $(\boldsymbol{g} , \boldsymbol{t}) \mapsto f \ast P_{\boldsymbol{t}}(\boldsymbol{g})$ is harmonic in the $g_1$ and $t_1$ variables, and in the $g_2$ and $t_2$ variables.
This leads to a more complicated analogue of \eqref{eq:harmonicity-1p-bis}, with four terms, $\mathrm{I}_{11}$, $\mathrm{I}_{21}$, $\mathrm{I}_{12}$, and $\mathrm{I}_{22}$, where the subscript $i_1i_2$ indicates a term like $\mathrm{I}_{i_1}$ in the first factor, and a term like $\mathrm{I}_{i_2}$ in the second factor.

There is one ``main term'' $\mathrm{I}_{11}$ with a double sub-Laplacian, namely,
\[
\slashed{\mathcal{L}}_1 \slashed{\mathcal{L}}_2 \Bigl( \left| f\ast P_{\boldsymbol{t}}(\boldsymbol{g})
\eta( \mathrm{H}_{\boldsymbol{t}}(\boldsymbol{g}))\right| ^{2} \Bigr)  .
\]
When integrated, by iterating the argument used to treat $\mathrm{I}_1$ in Section \ref{ssec:LP-and-max}, $\mathrm{I}_{11}$ gives
\[
\int_{L_1(\alpha)} \left| f(\boldsymbol{g}) \right| ^2 \,\mathrm{d}\boldsymbol{g} .
\]
The ``mixed terms'' $\mathrm{I}_{21}$ and $\mathrm{I}_{12}$, with a sub-Laplacian in one variable and a square function in the other, may be treated by \eqref{eq:2-parameter-case} and its analogue with $G_1$ and $G_2$ interchanged.
Finally, the ``double error term'' $\mathrm{I}_{22}$ may be treated using Littlewood--Paley theory, much as we treated $\mathrm{I}_2$ before.

\subsection{A reproducing formula}

We shall use the discrete Calder\'on reproducing formula from \cite[Theorem 2.9]{HLL}.
We first give a definition of the space ${\mathsf{M}}(G,r,g)$ of molecules of scale $r$ near a point $g$ on a group $G$ of homogeneous dimension $Q$, and then define the analogous space on a product group.
In the more general setting of spaces of homogenous type, this space was introduced in \cite{HLW}.
Recall that $\varepsilon \in (0,1]$ is a fixed parameter.

\begin{definition}\label{pre test}
Fix $r> 0$ and $g\in G$.
We say that a function $f$ on $G$ is in ${\mathsf{M}}(G,r,g)$ if there is a constant $C$ such that
\begin{equation}\label{eq05}
\begin{aligned}
\left| f(h) \right|
&\leq C\frac {r^\varepsilon}{(r+\rho(g^{-1}h))^{Q+\varepsilon}}
\\
\left| f(h) -f(h') \right|
&\le C
\frac{\rho(h^{-1}h')^\varepsilon}
{(r+ \rho(g^{-1}h)+ \rho(g^{-1}h'))^{Q+\varepsilon}}
\end{aligned}
\end{equation}
for all $h, h' \in G$.
If moreover $f$ satisfies the cancellation condition
\[
\int_G f(g)\,\mathrm{d}g = 0,
\]
then we write $f \in {\mathsf{M}_0}(G,r,g)$.
The norm $\left\Vert  f\right\Vert_{\mathsf{M}(G,r,g)}$ is defined to be the least constant $C$ such that the inequalities \eqref{eq05} both hold.
\end{definition}

Clearly, the space $\mathsf{M}(G)$ of \eqref{eq:molecule} is equal to $\mathsf{M}(G, 1,o)$, and $f \in \mathsf{M}(G)$ if and only if $f \in \mathsf{M}(G, r, g)$ for all $r > 0$ and all $g \in G$.
Changing $r$ or $g$ changes the norms.

We now define the molecular space $\mathsf{M}(\boldsymbol{G},\boldsymbol{r},\boldsymbol{g})$ on the product group $\boldsymbol{G}$ as follows.

\begin{definition}\label{product test-1}
Fix $\boldsymbol{r} \in \boldsymbol{T}$ and $\boldsymbol{g} \in \boldsymbol{G}$.
We say that $\psi:\boldsymbol{G} \to \mathbb{C}$ is in $\mathsf{M}(\boldsymbol{G},\boldsymbol{r},\boldsymbol{g})$ if
$\psi(\cdot,h_2)\in \mathsf{M}(G_1,r_1,g_1)$ for all $h_2 \in G_2$
and $\psi(h_1,\cdot)\in \mathsf{M}(G_2,r_2,g_2)$ for all $h_1 \in G_1$, and
\begin{equation}\label{eq06}
\begin{gathered}
\left\Vert  \psi(\cdot,h_2) \right\Vert_{{\mathsf{M}}(G_1,r_1,g_1)}
\leq C\frac{ r_2^\varepsilon}{(r_2+\rho_2(g_2^{-1}h_2) )^{Q_2+\varepsilon}}
\\
\left\Vert  \psi(h_1,\cdot) \right\Vert_{\mathsf{M}(G_2,r_2,g_2)}
\leq C\frac{ r_1^\varepsilon}{(r_1+\rho_1(g_1^{-1}h_1))^{Q_1+\varepsilon}}
\\
\left\Vert  \psi(\cdot,h_2)-\psi(\cdot,h'_2) \right\Vert_{{\mathsf{M}}(G_1,r_1,g_1)}
\leq C \frac{ \rho_2(h_2^{-1}h'_2)^\varepsilon }{(r_2 + \rho_2(g_2^{-1} h_2) + \rho_2(g_2^{-1}h'_2))^{Q + \varepsilon}}
\\
\left\Vert  \psi(h_1,\cdot)-\psi(h'_1,\cdot) \right\Vert_{{\mathsf{M}}(G_2,r_2,g_2)}
\leq C \frac{ \rho_1(h_1^{-1}h'_1)^\varepsilon }{(r_1 + \rho_1(g_1^{-1}h_1) + \rho_1(g_1^{-1}h'_1))^{Q + \varepsilon}}
\end{gathered}
\end{equation}
for all $\boldsymbol{h}, \boldsymbol{h}' \in \boldsymbol{G}$.
If moreover $\psi$ satisfies the cancellation conditions
\[
\int_{G_1} \psi(g_1, \cdot)\,\mathrm{d}g_1 = 0
\qquad\text{and}\qquad
\int_{G_2} \psi(\cdot, g_2)\,\mathrm{d}g_2 = 0,
\]
then we write $\psi \in {\mathsf{M}_0}(\boldsymbol{G},\boldsymbol{r},\boldsymbol{g})$.
The norm $\left\Vert  \psi\right\Vert_{\mathsf{M}(\boldsymbol{G}, \boldsymbol{r}, \boldsymbol{g})}$ is defined to be the least constant $C$ such that the inequalities \eqref{eq06} above all hold.
\end{definition}

Evidently, if $\psi_1\in {\mathsf{M}_0}(G_1,r_1,g_1)$ and $\psi_2\in {\mathsf{M}_0}(G_2,r_2,g_2)$, then $\psi_1 \otimes \psi_2 \in {\mathsf{M}_0}(\boldsymbol{G},\boldsymbol{r},\boldsymbol{g})$.
It is easy to check that
\[
\left\Vert  f \right\Vert_{\mathsf{M}(\boldsymbol{G},\boldsymbol{r},\boldsymbol{g})}
= \left\Vert  r_1^{Q_1} r_2^{Q_2}\psi(\boldsymbol{g} \delta_{\boldsymbol{r}}(\cdot)) \right\Vert_{\mathsf{M}(\boldsymbol{G})}.
\]
Hence the $\mathsf{L}^1(\boldsymbol{G})$ norms of elements of a bounded subset of ${\mathsf{M}}(\boldsymbol{G},\boldsymbol{r},\boldsymbol{g})$ are bounded.

We are now ready to state the version of the Calder\'on reproducing formula that we are going to use.
Let $\sigma: \mathscr{P}(\boldsymbol{G}) \to \boldsymbol{G}$ be an arbitrary function such that $\sigma(R) \in \bar R$ for all $R \in \mathscr{P}(\boldsymbol{G})$ and let $\boldsymbol{\ell}: \mathscr{P}(\boldsymbol{G}) \to \boldsymbol{T}$ be the function such that $\ell_i(Q_1 \times Q_2) = \ell(Q_i)$, the ``side-length'' of $Q_i$.
For $\varphi^{[1]} \in \mathsf{M}(G_1)$ and $\varphi^{[2]} \in \mathsf{M}(G_2)$, we take $\varphi$ to be $\varphi^{[1]} \otimes \varphi^{[2]}$, and define $\varphi_R$ to be the function $\boldsymbol{h} \mapsto [\bar\varphi]_{\boldsymbol{\ell}(R)}(\boldsymbol{h}^{-1}\sigma(R))$.
Observe that
\begin{equation}\label{eq:eq}
\int_{\boldsymbol{G}} f(\boldsymbol{h})  \bar\varphi_{R}(\boldsymbol{h}) \,\mathrm{d}\boldsymbol{h}
=\int_{\boldsymbol{G}} f(\boldsymbol{h}) \left[\varphi\right]_{\boldsymbol{\ell}(R)}(\boldsymbol{h}^{-1} \sigma(R)) \,\mathrm{d}\boldsymbol{h}
= f \ast [\varphi]_{\boldsymbol{\ell}(R)}(\sigma(R)).
\end{equation}

The point of the following theorem is that the collection $\{ \left| R\right| ^{1/2} \varphi_R: R \in \mathscr{P}(\boldsymbol{G})\}$ is a well-behaved frame in $\mathsf{L}^2(\boldsymbol{G})$, with a well-behaved dual frame $\{ \left| R\right| ^{1/2} \tilde\varphi_R : R \in \mathscr{P}(\boldsymbol{G})\}$.
By well-behaved, we mean that $\varphi_R$ and $\tilde\varphi_R$ are concentrated near $R$, and certain molecular norms of $\varphi_R$ and $\tilde\varphi_R$ are uniformly bounded in $R$ \emph{and} in $\sigma$.

\begin{theorem}[{\cite[Theorem 2.9]{HLL}}]\label{dcrf-22}
Suppose that $\sigma: \mathscr{P}(\boldsymbol{G}) \to \boldsymbol{G}$ and $\varphi \in \mathsf{M}(\boldsymbol{G})$ are as discussed above.
Then, after possible replacing $\varphi$ by a normalised dilate of $\varphi$, there exist functions $\tilde\varphi_R$ in $\mathsf{M}_0(\boldsymbol{G})$, which may also depend on $\sigma$, such that
\begin{align*}%\label{discrf-2}
\psi = \sum_{R \in \mathscr{P}(\boldsymbol{G})} \left| R\right|  \left\langle \psi , \varphi_{R} \right\rangle \tilde\varphi_R
\end{align*}
for every $\psi$ in $\mathsf{M}_0(\boldsymbol{G})$.
Further,
\[
\bigl\Vert \varphi_R \bigr\Vert_{ {\mathsf{M}}(\boldsymbol{G},\boldsymbol{\ell}(R),\sigma(R))}
 + \bigl\Vert \tilde\varphi_R \bigr\Vert_{ {\mathsf{M}_0}(\boldsymbol{G},\boldsymbol{\ell}(R),\sigma(R))}
\]
is uniformly bounded, irrespective of $R$ and the choice of $\sigma(R)$.
\end{theorem}

As $\mathsf{L}^1(\boldsymbol{G})$ is a subspace of the dual space of $\mathsf{M}_0(\boldsymbol{G})$, by \eqref{eq:eq} and a duality argument,
\begin{align*}%\label{discrf1}
f \ast \psi (\boldsymbol{g})
=\sum_{R \in \mathscr{P}(\boldsymbol{G})}
\left( \left| R\right|  f \ast [\varphi]_{\boldsymbol{\ell}(R)}(\sigma(R)) \right)
\tilde\varphi_{R} \ast \psi (\boldsymbol{g})
\qquad\forall \boldsymbol{g} \in \boldsymbol{G}
\end{align*}
for all $\psi \in \mathsf{M}_0(\boldsymbol{G})$ and all $f\in \mathsf{L}^1(\boldsymbol{G})$.

\subsection{Proof of Proposition \ref{prop equiv max}}

We are going to prove Proposition \ref{prop equiv max}.
Again, we need to extend what we know about homogeneous groups to product groups.
We shall prove a stronger result concerning the grand maximal function, which we now introduce.

We define
\[
\mathsf{F}(\boldsymbol{G})
:= \left\{ \zeta^{[1]} \otimes \zeta^{[2]} :
\left\Vert  \zeta^{[1]} \right\Vert_{\mathsf{M}(G_1)} \leq 1,
\left\Vert  \zeta^{[2]} \right\Vert_{\mathsf{M}(G_2)} \leq 1  \right\},
\]
and the grand maximal operator $\mathcal{G}$:
\begin{equation*}
\mathcal{G}(f)(\boldsymbol{g})
:=\sup\bigl\{ \left| f \ast \zeta_{\boldsymbol{t}} (\boldsymbol{g}) \right| :
\zeta \in \mathsf{F}(\boldsymbol{G}),
\boldsymbol{t} \in \boldsymbol{T} \bigr\}
\qquad\forall \boldsymbol{g} \in \boldsymbol{G}
\end{equation*}
for all $f\in \mathsf{L}^1(\boldsymbol{G})$.
We write $\mathcal{R}_{\boldsymbol{h}}$ for the operator of right translation by $\boldsymbol{h} \in \boldsymbol{G}$, that is, $\mathcal{R}_{\boldsymbol{h}}\zeta(\boldsymbol{g}) = \zeta(\boldsymbol{g}\boldsymbol{h})$ for all $\boldsymbol{g} \in \boldsymbol{G}$ and $\zeta \in \mathsf{F}(\boldsymbol{G})$.
Since
\[
\begin{aligned}
\mathcal{M}_{\zeta,\beta}(f) (\boldsymbol{g})
&= \sup\bigl\{ \left| f \ast \zeta_{\boldsymbol{t}} (\boldsymbol{g}') \right| :
\boldsymbol{g}' \in R(\boldsymbol{g}, \beta\boldsymbol{t}) ,
\boldsymbol{t} \in \boldsymbol{T} \bigr\}  \\
&= \sup\bigl\{ \left| f \ast (\mathcal{R}_{\boldsymbol{h}}\zeta)_{\boldsymbol{t}} (\boldsymbol{g}) \right| :
\boldsymbol{h} \in R(\boldsymbol{o}, \beta,\beta) ,
\boldsymbol{t} \in \boldsymbol{T} \bigr\} ,
\end{aligned}
\]
and $\mathcal{R}_{\boldsymbol{h}}\zeta$ is a uniformly bounded ($\beta$-dependent) multiple of a function in $\mathsf{F}(\boldsymbol{G})$ when $\boldsymbol{h} \in R(\boldsymbol{o},\beta,\beta)$ we deduce that
\begin{align*}
\mathcal{M}_{\zeta,\beta} f(\boldsymbol{g})
\lesssim_\beta \mathcal{G}(f)(\boldsymbol{g})
\qquad\forall \boldsymbol{g} \in \boldsymbol{G} .
\end{align*}

Take functions $\varphi^{[i]}$ on $G_i$ such that $\left\Vert  \varphi^{[i]} \right\Vert_{\mathsf{M}(G_i)} \leq 1$ and $\int_{G_i} \varphi^{[i]} \,\mathrm{d}g_i \neq 0$.
From the discussion above, it will suffice to prove that there exists $\theta \in (0,1)$ such that
\begin{align}\label{product domination}
\mathcal{M}_{\zeta,0} f(\boldsymbol{g})
\lesssim \Bigl( \mathcal{M}_S \bigl( \left| \mathcal{M}_{\varphi,0} (f)(\boldsymbol{g}) \right| ^\theta \bigr) \Bigr)^{1/\theta}
\qquad\forall \boldsymbol{g} \in \boldsymbol{G}
\end{align}
for all $\beta \geq 0$, all $\zeta \in \mathsf{F}(\boldsymbol{G})$, and all $f\in \mathsf{L}^1( \boldsymbol{G})$, for then the $\mathsf{L}^{1/\theta}(\boldsymbol{G})$ boundedness of $\mathcal{M}_S$ shows that
\[
\begin{aligned}
\left\Vert \mathcal{G}(f)\right\Vert_{\mathsf{L}^1(\boldsymbol{G})}
&\lesssim \int_{\boldsymbol{G}} \Bigl(  \mathcal{M}_S \bigl( \left| \mathcal{M}_{\varphi,0} (f) \right| ^\theta \bigr) \Bigr)^{1/\theta} \,\mathrm{d}\boldsymbol{g}
%\\&
\lesssim
\left\Vert  \mathcal{M}_{\varphi,0} (f) \right\Vert_{\mathsf{L}^1(\boldsymbol{G})} ,
\end{aligned}
\]
which implies the required result.
We may assume that $f$ is continuous, by mollification.

To prove \eqref{product domination}, we make and confirm three claims, which together imply the result.

\emph{Claim 1}: for $\theta$ less than but close to $1$,
\begin{align}\label{moser 1}
\left| f \ast \zeta_{\boldsymbol{t}} \right|
&\lesssim_\theta \left\Vert   \psi  \right\Vert_{{\mathsf{M}_0}(\boldsymbol{G})}
\bigl(  \mathcal{M}_S \bigl( \left| \mathcal{M}_{\varphi,0}(f) \right| ^\theta \bigr) \bigr)^{1/\theta}
\end{align}
for all $f\in \mathsf{L}^1( \boldsymbol{G})$, all $\boldsymbol{t} \in \boldsymbol{T}$ and all $\zeta$ of the form $\psi_1 \otimes \psi_2$, where $\psi_1\in {\mathsf{M}_0}(G_1)$ and $\psi_2\in {\mathsf{M}_0}(G_2)$.
To prove this claim, we use Theorem \ref{dcrf-22}, which tells us that
\begin{align}\label{discrf2}
f \ast \psi_{\boldsymbol{t}}(\boldsymbol{g})
= \sum_{R \in \mathscr{P}(\boldsymbol{G})}\left| R\right|  \left\langle f , \varphi_{R} \right\rangle
{\tilde\varphi}_{R}\ast \psi_{\boldsymbol{t}}(\boldsymbol{g}) .
\end{align}

On the one hand, once $f$ is given, we may choose $\sigma$ such that
\[
\left| \left\langle f , \varphi_{R} \right\rangle\right|
= \left| (f \ast [\varphi]_{\boldsymbol{\ell}(R)}) (\sigma(R))\right|
= \min\{ \left| (f \ast [\varphi]_{\boldsymbol{\ell}(R)} ) (\boldsymbol{h}) \right|  : \boldsymbol{h} \in \bar R \} ,
\]
and on the other, by the almost orthogonality estimate of \cite[(4.4)]{HLW}, for all choices of $\sigma$ and all choices of $\boldsymbol{h}$ in $R$,
\[
\begin{aligned}
\left|  {\tilde\varphi}_{R} \ast \psi_{\boldsymbol{t}}(\boldsymbol{g}) \right|
&\lesssim \mu_{\boldsymbol{j},\boldsymbol{t}}({\boldsymbol{h}}^{-1}\boldsymbol{g}) \\
&\eqsim
\frac { (\ell_1(R) \wedge t_1)^{\varepsilon} (\ell_1(R) t_1)^{-\varepsilon}} {(\ell_1(R) \wedge t_1)^{-1} + \rho_1(h_1^{-1} g_1))^{Q_1+\varepsilon}}
\frac { (\ell_2(R) \wedge t_2)^{\varepsilon} (\ell_2(R) t_2)^{-\varepsilon}} {(\ell_2(R) \wedge t_2)^{-1} + \rho_1(h_2^{-1} g_2))^{Q_2+\varepsilon}} .
\end{aligned}
\]
Here $\mu_{\boldsymbol{j},\boldsymbol{t}}$ is the least decreasing biradial majorant for all the functions ${\tilde\varphi}_{R} \ast \psi_{\boldsymbol{t}}(\boldsymbol{h}^{-1}\cdot)$ when $\boldsymbol{h} \in R \in \mathscr{P}^{\boldsymbol{j}}(\boldsymbol{G})$, and  $a \wedge b$ denotes the minimum of $a$ and $b$.
Recall that the ``sidelengths'' of the cubes making the rectangle $R \in \mathscr{P}^{\boldsymbol{j}}(\boldsymbol{G})$ are $\kappa^{j_1}$ and $\kappa^{j_2}$.
Then, by also using \eqref{discrf2} and \eqref{eq:decay-max}, we see that
\begin{align}
\left|  f \ast \psi_{\boldsymbol{t}}(\boldsymbol{g}) \right| ^\theta
&\leq \biggl( \sum_{\boldsymbol{j} \in \mathbb{Z}^2}\sum_{R\in\mathscr{P}^{\boldsymbol{j}}(\boldsymbol{G})}  \left| R\right| \min_{\boldsymbol{g} \in \bar R}
\mathcal{M}_{\varphi,0}(f)( {\boldsymbol{g}}) \left|  {\tilde\varphi}_{R} \ast \psi_{\boldsymbol{t}}(\boldsymbol{g}) \right|  \biggr)^\theta
\notag\\
&\leq \biggl( \sum_{\boldsymbol{j} \in \mathbb{Z}^2}\sum_{R\in\mathscr{P}^{\boldsymbol{j}}(\boldsymbol{G})}  \left| R\right|^\theta \biggl( \min_{\boldsymbol{g} \in \bar R}
\mathcal{M}_{\varphi,0}(f)( {\boldsymbol{g}}) \biggr)^\theta
\left|  {\tilde\varphi}_{R} \ast \psi_{\boldsymbol{t}}(\boldsymbol{g}) \right| ^\theta \biggr)
\notag\\
&= \biggl( \sum_{\boldsymbol{j} \in \mathbb{Z}^2}\sum_{R\in\mathscr{P}^{\boldsymbol{j}}(\boldsymbol{G})}  \left| R\right|^{\theta-1} \int_{R} \biggl( \min_{\boldsymbol{g} \in \bar R}
\mathcal{M}_{\varphi,0}(f)( {\boldsymbol{g}}) \biggr)^\theta
\left|  {\tilde\varphi}_{R} \ast \psi_{\boldsymbol{t}}(\boldsymbol{g}) \right| ^\theta \,\mathrm{d}\boldsymbol{h}\biggr)
\label{eeeee1}\\
&\lesssim \biggl( \sum_{\boldsymbol{j} \in \mathbb{Z}^2} (\kappa_1^{j_1Q_1}\kappa_2^{j_2Q_2})^{\theta-1}
\int_{\boldsymbol{G}} \biggl(
\mathcal{M}_{\varphi,0}(f)(\boldsymbol{h}) \biggr)^\theta \mu_{\boldsymbol{j},\boldsymbol{t}}(\boldsymbol{h}^{-1}\boldsymbol{g})^\theta \,\mathrm{d}\boldsymbol{h}\biggr)
\notag\\
&\leq \biggl( \sum_{\boldsymbol{j} \in \mathbb{Z}^2} (\kappa_1^{j_1Q_1}\kappa_2^{j_2Q_2})^{\theta-1}
\left\Vert  \mu_{\boldsymbol{j},\boldsymbol{t}}^\theta\right\Vert_{\mathsf{L}^1(\boldsymbol{G})} \mathcal{M}_{S} \left(
\mathcal{M}_{\varphi,0}(f) \right)^\theta (\boldsymbol{g}) \biggr).
\notag
\end{align}

If $\max\{ Q_1/(Q_1+\varepsilon), Q_2/(Q_2 + \varepsilon)\}  < \theta < 1$, then computation shows that
\begin{align*}
\sum_{\boldsymbol{j} \in \mathbb{Z}^2} (\kappa_1^{j_1Q_1}\kappa_2^{j_2Q_2})^{\theta-1}
\left\Vert  \mu_{\boldsymbol{j},\boldsymbol{t}}^\theta\right\Vert_{\mathsf{L}^1(\boldsymbol{G})} < \infty.
\end{align*}
Thus the right-hand side of \eqref{eeeee1} is bounded by a multiple of
\[
\bigl( \mathcal{M}_S \bigl(\left| \mathcal{M}_{\varphi,0}(f) \right| ^\theta\bigr)\bigr) ( \boldsymbol{g}^R),
\]
which implies \eqref{moser 1} and proves our claim.

\emph{Claim 2}: for $\theta$ less than but close to $1$,
\begin{align*}%\label{moser 1-2}
\left| f \ast \zeta_{\boldsymbol{t}}(\boldsymbol{g}) \right|
&\lesssim_\theta \left\Vert   \zeta_1  \right\Vert_{{\mathsf{M}_0}(G_1)}
\bigl( \mathcal{M}_S \bigl( \left| \mathcal{M}_{\varphi,0}(f) \right| ^\theta \bigr) (\boldsymbol{g}) \bigr)^{1/\theta}
\end{align*}
for all $f\in \mathsf{L}^1( \boldsymbol{G})$, all $\boldsymbol{t} \in \boldsymbol{T}$, where $\zeta = \psi_1 \otimes \varphi_2$; here $\psi_1 \in {\mathsf{M}_0}(G_1)$.

The proof of this claim involves use of a reproducing formula that involves the first variable only, namely,
\[
f^{[1]} \ast_1 \psi^{[1]}_{t_1}(g_1)
= \sum_{Q \in \mathscr{Q}(G_1)}\left| Q\right|
f^{[1]} \ast_1 [\varphi^{[1]}]_{\ell(Q)}(\sigma(Q))
\tilde\varphi^{[1]}_{Q} \ast \psi^{[1]}_{t_1}(g_1)
\qquad\forall g_1 \in G_1
\]
where $f^{[1]} \in \mathsf{L}^1(G_1)$.
This implies that
 \[
f \ast \zeta_{\boldsymbol{t}}(\boldsymbol{g})
= \sum_{Q \in \mathscr{Q}(G_1)}\left| Q\right|
f \ast \bigl( [\varphi^{[1]}]_{\ell(Q)}\otimes\varphi^{[2]}_{t_2}\bigr)(\sigma(Q),g_2)
\tilde\varphi^{[1]}_{Q} \ast \psi^{[1]}_{t_1}(g_1)
\qquad\forall \boldsymbol{g} \in \boldsymbol{G},
\]
when $f \in \mathsf{L}^1(\boldsymbol{G})$; a similar argument to that for Claim 1 may be used.
We see that
\begin{equation*}%\label{eeeee1-2}
\begin{aligned}
\left|  {\tilde\varphi}_{R} \ast \zeta_{\boldsymbol{t}}(\boldsymbol{g}) \right| ^\theta
&\leq \biggl( \sum_{j_1 \in \mathbb{Z}}\sum_{Q\in\mathscr{Q}^{j_1}(G_1)}
\left| Q\right| \min_{g_1 \in \bar Q}
\mathcal{M}_{\varphi,0}(f)( {g_1,g_2}) \left|  {\tilde\varphi}^{[1]}_{Q} \ast \psi^{[1]}_{t_1}(g_1) \right|  \biggr)^\theta \\
&\leq \biggl( \sum_{j_1 \in \mathbb{Z}} (\kappa_1^{j_1Q_1})^{\theta-1}
\left\Vert  \bigl(\mu^{[1]}_{j_1,t_1}\bigr)^\theta\right\Vert_{\mathsf{L}^1(G_1)}
(\mathcal{M}_{1} \otimes \mathcal{I}_2) \left(
|\mathcal{M}_{\varphi,0}(f)|^\theta \right) (\boldsymbol{g}) \biggr) \\
&\lesssim
\mathcal{M}_{S} \left(
|\mathcal{M}_{\varphi,0}(f)|^\theta \right) (\boldsymbol{g})  ,
\end{aligned}
\end{equation*}
as claimed; here $\mathcal{I}_2$ denotes the identity operator acting on functions on $G_2$.

\emph{Claim 3}: for $\theta$ less than but close to $1$,
\begin{align*}%\label{moser 102}
\left| f \ast \psi_{\boldsymbol{t}}(\boldsymbol{g}) \right|
&\lesssim_\theta \left\Vert   \psi_2  \right\Vert_{{\mathsf{M}_0}(G_2)}
\bigl(  \mathcal{M}_S \bigl( \left| \mathcal{M}_{\varphi,0}(f) \right| ^\theta \bigr) (\boldsymbol{g}) \bigr)^{1/\theta}
\end{align*}
for all $f\in \mathsf{L}^1( \boldsymbol{G})$, all $\boldsymbol{t} \in \boldsymbol{T}$, where now $\psi_1 = \varphi_1$ while $\psi_2\in {\mathsf{M}_0}(G_2)$.
The proof of this is a very minor modification of that of Claim 2.

To finish the proof, we must estimate $\mathcal{M}_{\zeta,0}f$, where $\zeta_1 \in \mathsf{M}(G_1)$ and $\zeta_2 \in \mathsf{M}(G_2)$.
We write $\zeta_1  = c_1 \varphi_1 + \psi_1$, where $\psi \in \mathsf{M}_0(G_1)$ and $c_1$ is chosen to make the integrals of both sides equal, and we decompose $\zeta_2$ analogously.
Then $\mathcal{M}_{\zeta,0}f(\boldsymbol{g})$ is dominated by a sum of four terms,
each of which is bounded pointwise by $(\mathcal{M}_S(\mathcal{M}_{\varphi,0}f)^\theta)^{1/\theta}(\boldsymbol{g})$.
This proves the desired inequality and hence Proposition \ref{prop equiv max}.

\section{Concluding remarks}

Many of our results can be proved in greater generality.
For example, Proposition \ref{prop equiv max} should be true on much more general spaces of homogeneous type.
Other results require the structure of stratified group that we have used here.
These include Theorem \ref{thm Riesz} and Proposition \ref{prop poisson L-P and max}.
Indeed, the first relies on the Christ--Geller singular integral characterisation of the Hardy space on stratified groups, and the second on various properties of the Poisson kernel.
It is an interesting challenge to extend either of these to a more substantial class of nilpotent Lie groups, let alone to general spaces of homogeneous type.

\section{Thanks}
The authors thank the unknown referees for their careful reading of and helpful comments about the paper.

J. Li and M. Cowling are supported by the Australian Research Council (ARC) through
research grant DP220100285.
L. Yan was supported by the NNSF of China, Grant No.~11871480, and by the Australian Research Council (ARC) through research grant DP190100970.

\end{document}